\documentclass{aptpub}

\authornames{S.~Engelke, A.~Malinowski, M.~Oesting, M.~Schlather} 
\shorttitle{Representations of max-stable processes based on single extreme events}
\usepackage[usenames,dvipsnames,svgnames,table]{xcolor}
\usepackage{amsmath} 
\RequirePackage[colorlinks,citecolor=blue,urlcolor=blue]{hyperref}
\RequirePackage{hypernat}
\usepackage{amssymb}
\usepackage{ifthen}
\usepackage{enumitem}
\usepackage{fixme}
\usepackage{grffile}

\newcommand{\R}{\mathbb R}

\newcommand{\Nb}{\mathbb N}
\newcommand{\Z}{\mathbb Z}

\newcommand{\1}{\mathbf 1}
\newcommand{\0}{\mathbf 0}

\newcommand{\Q}{Q}

\newcommand{\Xbf}{\mathbf X}
\newcommand{\Wbf}{\mathbf W}
\newcommand{\Vbf}{\mathbf V}

\newcommand{\sbf}{\mathbf s}

\newcommand{\xbf}{\mathbf x}
\newcommand{\zbf}{\mathbf z}
\newcommand{\Scal}{\mathcal S}

\DeclareMathOperator{\supp}{supp}

\DeclareMathOperator{\discr}{discr}

\DeclareMathSymbol{\optimes}{\mathbin}{symbols}{"02}

\renewcommand{\vec}{\mathbf}
\DeclareMathOperator*{\argmax}{arg\,max}
\DeclareMathOperator*{\argsup}{arg\,sup}

\begin{document}

\title{Representations of max-stable processes\\ based on single extreme events}

\authorone[Georg-August-Universit\"at G\"ottingen]{Sebastian~Engelke}
\authortwo[Georg-August-Universit\"at G\"ottingen]{Alexander~Malinowski}
\authorthree[Universit\"at Mannheim]{Marco~Oesting}
\authorfour[Universit\"at Mannheim]{Martin~Schlather}
\addressone{Institut f\"ur Mathematische Stochastik,               
              Goldschmidtstr.~7, 37077 G\"ottingen, Germany,
              \texttt{sengelk@uni-goettingen.de}}
\addresstwo{Institut f\"ur Mathematische Stochastik,               
              Goldschmidtstr.~7, 37077 G\"ottingen, Germany, 
              \texttt{malinows@math.uni-goettingen.de}}
\addressthree{Institut f\"ur Mathematik, Universit\"at Mannheim, 
              A 5, 6, 68131 Mannheim, Germany, 
              \texttt{oesting@math.uni-mannheim.de}}              
\addressfour{Institut f\"ur Mathematik, Universit\"at Mannheim, 
              A 5, 6, 68131 Mannheim, Germany, 
              \texttt{schlather@math.uni-mannheim.de}}

\begin{abstract}
This paper provides the basis for new methods of inference for
max-stable processes $\xi$ on general spaces that admit a certain
incremental representation, which, in important cases, has a much
simpler structure than the max-stable process itself. 
A corresponding 
peaks-over-threshold approach will incorporate
all single events that are extreme in some sense
and will therefore rely on a substantially larger amount
of data in comparison to estimation procedures based on block
maxima. \\ Conditioning a process $\eta$ in the max-domain of
attraction of $\xi$ on being \emph{extremal}, several convergence
results for the increments of $\eta$ are proved.  In a similar way,
the shape functions of mixed moving maxima (M3) processes can be
extracted from suitably conditioned single events $\eta$.  Connecting
the two approaches, transformation formulae for processes that
admit both an incremental and an M3 representation are identified.

\end{abstract}

\section{Introduction}

The joint extremal behavior at multiple locations of some random
process $\{\eta(t): t\in T\}$, $T$ an arbitrary index set, can be
captured via its limiting \emph{max-stable process}, assuming the
latter exists and is non-trivial everywhere. Then, for independent
copies $\eta_i$ of $\eta$, $i\in\Nb$, the functions $b_n: T \to \R$,
$c_n : T\to (0,\infty)$ can be chosen such that the convergence
\begin{align}\label{MDA}
  \xi(t) = \lim_{n\to\infty} c_n(t) \Big(\max_{i=1}^n \eta_i(t) - b_n(t)\Big), 
  \quad t\in T,
\end{align}
holds in the sense of finite-dimensional distributions. 
The process $\xi$ is said
to be \emph{max-stable} and $\eta$ is in its max-domain of attraction
(MDA).  The theory of max-stable processes is mainly concerned with the
dependence structure while the marginals are usually assumed to be
known. Even for finite-dimensional max-stable distributions, the space
of possible dependence structures is uncountably infinite-dimensional
and parametric models are required to find a balance between
flexibility and analytical tractability \cite{deh2006a,res2008}.

A general construction principle for max-stable processes was provided
by \cite{ deh1984,smi1990}: Let $\sum_{i\in\Nb} \delta_{(U_i, S_i)}$
be a Poisson point process (PPP) on $(0,\infty)\times\Scal$ with
intensity measure $u^{-2}\rd u\cdot \nu(\rd s)$, where $(\Scal,
\mathfrak S)$ is an arbitrary measurable space and $\nu$ a positive
measure on $\Scal$. Further, let $f:\Scal\times T \to [0, \infty)$ be
  a non-negative function with $\int_{\Scal} f(s,t) \nu(\rd s) = 1$
  for all $t\in T$.  Then the process
\begin{align}
\xi(t) = \max_{i\in\Nb} U_i f(S_i, t), \quad t\in T,\label{constr_max_stable}
\end{align}
is max-stable and has standard Fr\'echet margins with distribution
function $\exp(-1/x)$ for $x \geq 0$.  In this paper, we restrict to
two specific choices for $f$ and $(\Scal, \mathfrak S, \nu)$ and
consider processes that admit one of the resulting representations.
First, let $\{W(t) : t\in T\}$ be a non-negative stochastic process
with $\sE W(t) = 1$, $t\in T$, and $W(t_0) = 1$ a.s.\ for some point
$t_0 \in T$.  The latter condition means that $W(t)$ simply describes
the multiplicative increment of $W$ w.r.t.\ the location $t_0$. For
$(\Scal, \mathfrak S, \nu)$ being the canonical probability space for
the sample paths of $W$ and with $f(w,t)=w(t)$, $w\in\Scal$, $t\in T$,
we refer to
\begin{align}
  \label{def_xi}
  \xi(t) = \max_{i\in\Nb} U_i W_i(t), \quad t\in T,
\end{align}
as the \emph{incremental representation} of $\xi$, where
$\{W_i\}_{i\in\Nb}$ are independent copies of $W$.  Since $T$ is an
arbitrary index set, the above definition covers multivariate extreme
value distributions, i.e. $T=\{t_1,\dots,t_k\}$, as well as max-stable
random fields, i.e. $T = \R^d$.\\ For the second specification, let
$\{F(t): \ t \in \R^d\}$ be a stochastic process with sample paths in
the space $C(\R^d)$ of non-negative continuous functions, such that
\begin{align}
  \label{assumption_integral}
  \textstyle \sE \int_{\R^d}  F(t) \rd t = 1.
\end{align}
With $S_i = (T_i,F_i)$, $i\in\Nb$, in $\Scal = \R^d\times C(\R^d)$,
intensity measure $\nu(\rd t \times \rd g)=\rd t\sP_F(\rd g)$ and
$f((t,g), s)=g(s-t)$, $(t,g)\in\Scal$, we obtain the class of
\emph{mixed moving maxima (M3) processes}
\begin{align}
  \xi(t) = \max_{i\in\Nb} U_i F_i(t- T_i), \quad t\in\R^d. \label{def_M3}
\end{align}
These processes are max-stable and stationary on $\R^d$ (see for
instance \cite{wan2010}).  The function $F$ is called \emph{shape
  function of $\xi$} and can also be deterministic (e.g., in case of
the Smith process).  In Smith's ``rainfall-storm'' interpretation
\cite{smi1990}, $U_i$ and $T_i$ are the strength and center point of
the $i$th storm, respectively, and $U_i F_i(t- T_i)$ represents the
corresponding amount of rainfall at location $t$. In this case,
$\xi(t)$ is the process of extremal precipitation.

When i.i.d.\ realizations $\eta_1, \ldots, \eta_n$ of $\eta$ in the
MDA of a max-stable process $\xi$ are observed, a classical approach
for parametric inference on $\xi$ is based on generating (approximate)
realizations of $\xi$ out of the data $\eta_1, \ldots, \eta_n$ via
componentwise block maxima and applying maximum likelihood (ML)
estimation afterwards.  A clear drawback of this method is that it
ignores all information on large values that is contained in the order
statistics below the within-block maximum. Further, ML estimation
needs to evaluate the multivariate densities while for many max-stable
models only the bivariate densities are known in closed form. Thus,
composite likelihood approaches have been proposed
\cite{pad2010,dav2012}. \\ In univariate extreme-value theory, the
second standard procedure estimates parameters by fitting a certain
PPP to the \emph{peaks-over-thresholds} (POT), i.e., to the empirical
process of exceedances over a certain critical value
\cite{lea1991,emb1997}.  Also in the multivariate framework we can
expect to profit from using all extremal data via generalized POT
methods instead of aggregated data.  In contrast to the ML approach,
in this paper, we assume that $\xi$ admits one of the two
representations \eqref{def_xi} and \eqref{def_M3} and we aim at
extracting realizations of the processes $W$ and $F$, respectively,
from \emph{single extreme events}. Here, the specification of a single
extreme event will depend on the respective representation.  \\ In
\cite{eng2012a}, this concept is applied to derive estimators for the
class of Brown-Resnick processes \cite{bro1977, kab2009}, which have
the form \eqref{def_xi} by construction. With $a(n)$ being a sequence
of positive numbers with $\lim_{n\to\infty} a(n) = \infty$, the
convergence in distribution
\begin{align}
    \Bigg( \frac{\eta(t_1)}{\eta(t_0)}, \ldots, 
    \frac{\eta(t_k)}{\eta(t_0)} 
    \ \Bigg|\ 
    \eta(t_0) > a(n) \Bigg)\cvgdist \bigl(
    W(t_1),\dots, W(t_k) \bigr),
\label{cond_incr_conv}
\end{align}
$t_0,t_1,\dots,t_k\in T$, $k\in \Nb$, is established for $\eta$ being
in the MDA of a Brown-Resnick process and with $W$ being the
corresponding log-Gaussian random field.  A similar approach exists in
the theory of homogeneous discrete-time Markov chains. For instance,
\cite{seg2007} and \cite{ehl2011} investigate the behavior of a Markov
chain $\{M(t): t\in \Z\}$ conditional on the event that $M(0)$ is
large. The resulting extremal process is coined the tail chain and
turns out to be Markovian again.  In this paper, the convergence
result \eqref{cond_incr_conv} is generalized in different
aspects. Arbitrary non-negative processes $\{W(t) : t\in T\}$ with
$\sE W(t) = 1$, $t\in T$, are considered,  and convergence of the
conditional increments of $\eta$ in the sense of finite-dimensional
distributions as well as weak convergence in continuous function
spaces is shown (Theorems \ref{theo_cond_increments_general} and
\ref{theo_cond_increments_cont}).  Moreover, in Section
\ref{M3representation}, similar results are established for M3
processes \eqref{def_M3} by considering realizations of $\eta$ around
their (local) maxima.  Since one and the same max-stable process $\xi$
might admit both representations \eqref{def_xi} and \eqref{def_M3}
we provide formulae for switching between them in Section
\ref{sec:switching}. Section~\ref{sec:application} gives an exemplary
outlook on how our results can be applied for statistical inference.

\section{Incremental representation}
\label{examples_increment_representation}
Throughout this section, we suppose that $\{\xi(t): \ t\in T\}$, where
$T$ is an arbitrary index set, is normalized to standard Fr\'echet
margins and admits a representation
\begin{align}
  \label{def_xi2}
  \xi(t) = \max_{i\in\Nb} U_i V_i(t), \quad t\in T,
\end{align}
where $\sum_{i\in\Nb}\delta_{U_i}$ is a PPP on $(0,\infty)$ with
intensity $u^{-2}du$, which we call \emph{Fr\'echet point process} in
the following.  The $\{V_i\}_{i\in\Nb}$ are independent copies of a
non-negative stochastic process $\{V(t): \ t\in T\}$ with $\sE V(t) =
1$, $t\in T$.  Note that \eqref{def_xi2} is slightly less restrictive
than the representation \eqref{def_xi} in that we do not require that
$V(t_0)=1$ a.s.\ for some $t_0\in T$.  For any fixed $t_0\in T$, we
have
\begin{align}
  \label{decomp_V}
  \xi(t) \eqdist \max_{i\in\Nb} U_i 
  \left(\1_{P_i=0}V^{(1)}_i(t) + \1_{P_i=1}V^{(2)}_i(t)\right), 
  \quad t\in T,
\end{align}
where $\{P_i\}_{i\in\Nb}$ are i.i.d.\ Bernoulli variables with
parameter $p=\sP(V(t_0) = 0)$ and the $V^{(1)}_i$ and $V^{(2)}_i$ are
independent copies of the process $\{V(t): \ t\in T\}$, conditioned on
the events $\{V(t_0) > 0\}$ and $\{V(t_0)= 0\}$, respectively.

Note that for $k\in\Nb$, $t_0, \ldots, t_k\in T$, the vector $\Xi =
(\xi(t_0),\dots,\xi(t_k))$ follows a $(k+1)$-variate extreme-value
distribution and its distribution function $G$ can therefore be
written as
\begin{align}
  \label{def_mu}
  G(\mathbf{x}) = \exp( -\mu( [\0,\mathbf{x}]^C) ), 
  \quad \mathbf{x} \in \R^{k+1},
\end{align}
where $\mu$ is a measure on $E = [0,\infty)^{k+1}\setminus\{\0\}$, the
so-called \emph{exponent measure} of $G$ \cite[Prop.\ 5.8]{res2008}, 
and $[\0,\mathbf{x}]^C = E\setminus
  [\0,\mathbf{x}]$.

The following convergence result provides the theoretical
foundation for statistical inference based on the incremental process
$V$.

\begin{thm}\label{theo_cond_increments_general} 
  Let $\{\eta(t): \ t\in T\}$ be non-negative and in the MDA of some
  max-stable process $\xi$ that admits a representation
  \eqref{def_xi2} and suppose that $\eta$ is normalized such that
  \eqref{MDA} holds with $c_n(t) = 1/n$ and $b_n(t) = 0$ for $n\in\Nb$
  and $t\in T$.  Let $a(n)\to\infty$ as $n\to\infty$.  For $k\in \Nb$
  and $t_0,\dots,t_k\in T$ we have the convergence in distribution on
  $\R^{k+1}$
  \begin{align*}
    \left(\frac{\eta(t_0)}{a(n)}, \frac{\eta(t_1)}{\eta(t_0)} ,\dots, 
    \frac{\eta(t_k)}{\eta(t_0)} 
    \ \Bigg|\ 
    \eta(t_0) > a(n)\right) \cvgdist
    \left(Z, \Delta\mathbf{\tilde{V}}^{(1)}\right),\quad n\to \infty,
  \end{align*}
  where the distribution of $\Delta\mathbf{\tilde{V}}^{(1)}$ is given by
  \begin{align}
    \sP(\Delta\mathbf{\tilde{V}}^{(1)}\in d \mathbf z) = 
    (1-p)\sP(\Delta\mathbf V^{(1)}\in d \mathbf z)
    \sE\bigl( V^{(1)}(t_0) \big| 
    \Delta\mathbf V^{(1)}=\mathbf z \bigr), \quad \mathbf{z} \geq \0.
    \label{density_increment}
  \end{align}
  Here, $\Delta\mathbf V^{(1)}$ denotes the vector of increments
  $\left(\frac{V^{(1)}(t_1)}{V^{(1)}(t_0)}, \ldots,
  \frac{V^{(1)}(t_k)}{V^{(1)}(t_0)}\right)$ with respect to $t_0$, and
  $Z$ is an independent Pareto variable.
\end{thm}

\begin{rem}
  Note that any process $\eta$ that satisfies the convergence in
  \eqref{MDA} for a process $\xi$ with standard Fr\'echet margins can
  be normalized such that the norming functions in \eqref{MDA} become
  $c_n(t) = 1/n$ and $b_n(t) = 0$, $n\in\Nb$, $t\in T$
  \cite[Prop.\ 5.10]{res2008}.
\end{rem}

\begin{proof}[Proof of Theorem \ref{theo_cond_increments_general}]
  For $\mathbf{X} = (\eta(t_0),\dots,\eta(t_k))$, which is in the MDA
  of the random vector $\Xi=(\xi(t_0),\dots,\xi(t_k))$, it follows
  from \cite[Prop.\ 5.17]{res2008} that
\begin{align}
  \label{conv_resnick}
  \lim_{m\to\infty} m \sP( \mathbf{X}/m \in B ) = \mu(B),
\end{align} 
for all elements $B$ of the Borel $\sigma$-algebra $\mathcal B(E)$ of
$E$ bounded away from $\{\0\}$ with $\mu(\partial B)=0$, where $\mu$
is defined by \eqref{def_mu}.  For $s_0> 0$ and $\sbf =(s_1, \ldots,
s_k)\in [0, \infty)^{k}$, we consider the sets
  $A_{s_0}=(s_0,\infty)\times [0, \infty)^k$, $A=A_1$ and
    $B_{\mathbf{s}} = \{\mathbf{x} \in [0, \infty)^{k+1} :
      (x^{(1)},\dots,x^{(k)}) \leq x^{(0)}\mathbf{s}\}$ for $\sbf$
      satisfying $\sP( \Delta\tilde \Vbf^{(1)}\in \partial
      [\0,\sbf])=0$. Then
\begin{align*}
  \left\{ \eta(t_0) > s_0 a(n),\, 
  \big( \eta(t_1) / \eta(t_0) ,\dots, \eta(t_k) / \eta(t_0) \big)
  \leq \mathbf{s} \right\}
  = \{ \mathbf{X} / a(n) \in B_{\mathbf{s}}\cap A_{s_0} \},
\end{align*}
since $B_{\mathbf{s}}$ is invariant under
multiplication, i.e., $B_\sbf=cB_\sbf$ for any $c>0$.
Thus, we obtain
\begin{align}
  \notag
  \sP&\left( \eta(t_0) > s_0 a(n),
  \, \left( \eta(t_1) / \eta(t_0) ,\dots, 
  \eta(t_k) / \eta(t_0) \right) \leq \mathbf{s} \,\Big|\,
  \eta(t_0) > a(n) \right) \\
  \notag&= \frac{
    {a(n)} \sP( \mathbf{X} / a(n) 
    \in B_{\mathbf{s}} \cap A \cap A_{s_0} )}{
    {a(n)} \sP( \mathbf{X} / a(n) \in A)} \\
  \label{eq:01}
  & \longrightarrow 
  \frac{\mu(B_{\mathbf{s}} \cap A \cap A_{s_0})}{\mu(A)},\quad (n\to\infty),
\end{align}
where the convergence follows from \eqref{conv_resnick}, as long as
$\mu\{ \partial (B_{\mathbf{s}} \cap A \cap A_{s_0})\} = 0$.\\ Let
  \begin{align}
    \label{def_xi3}
    \xi^{(1)}(t) = \max_{i\in\Nb} U_i^{(1)} V^{(1)}_i(t),
    \quad t\in T,
  \end{align}
  where $\sum_{i\in\Nb} \delta_{U_i^{(1)}}$ is a Poisson point process
  with intensity $(1-p)u^{-2}\rd u$ and let $\mu^{(1)}$ be the
  exponent measure of the associated max-stable random vector
  $(\xi^{(1)}(t_0), \ldots, \xi^{(1)}(t_k))$.  Then the choice $A =
  (1,\infty)\times [0,\infty)^k$ guarantees that $\mu(\cdot \cap A) =
    \mu^{(1)}(\cdot \cap A)$.  Comparing the construction of
    $\xi^{(1)}$ in \eqref{def_xi3} with the definition of the exponent
    measure, we see that $\mu^{(1)}$ is the intensity measure of the
    Poisson point process $\sum_{i\in\Nb} \delta_{(U_i^{(1)}
      V_i^{(1)}(t_0),\, \ldots,\, U_i^{(1)} V_i^{(1)}(t_k))}$ on $E$.
    Hence,
  \begin{align}
    \mu(A) &= \int_0^\infty (1-p)u^{-2} \sP(u V^{(1)}(t_0) > 1) \rd u \notag\\    
    &= (1-p)\int_0^\infty u^{-2} \int_{[u^{-1}, \infty)} 
      \sP(V^{(1)}(t_0) \in \rd y) \rd u \notag\\
      &= (1-p)\int_0^\infty y \sP(V^{(1)}(t_0) \in \rd y)
      = (1-p)\sE V^{(1)}(t_0) = 1,
  \end{align}
  where the last equality follows from $\sE V^{(1)}(t_0) = \sE
  V(t_0)/(1-p)$.  Furthermore, for $s_0\geq 1$ and $\sbf\in[0,\infty)^k$
  with $\sP(\Delta\mathbf{\tilde V}^{(1)} \in \partial [\0,\sbf])=0$,
  \begin{align}
    &\mu(B_{\mathbf{s}} \cap A \cap A_{s_0}) / ((1-p)\mu(A)) \notag\\
    &= \int_0^\infty  u^{-2} 
    \sP\Bigl(u V^{(1)}(t_0) > s_0,\, 
    \big(u V^{(1)}(t_1),\dots, u V^{(1)}(t_k)\big) \leq 
    \mathbf{s} u V^{(1)}(t_0) \Bigr) \rd u\notag\\
    &=  
    \int_0^\infty  \int_{[s_0 u^{-1},\, \infty)}  u^{-2} 
      \sP\Bigl(V^{(1)}(t_0)\in \rd y 
      \Big|\Delta\mathbf V^{(1)} \leq \mathbf{s} \Bigr)
       \sP(\Delta\mathbf V^{(1)} \leq \mathbf{s} )\rd u
      \notag\\    
    &=  
    \int_{[\0, \mathbf s]} \int_{[0,\infty)} y s_0^{-1} \cdot 
      \sP\Bigl(V^{(1)}(t_0)\in \rd y \Big| 
      \Delta\mathbf V^{(1)}=\mathbf z \Bigr)
      \sP(\Delta\mathbf V^{(1)}\in \rd\zbf)
      \notag\\
      &=  
      s_0^{-1}\int_{[\0, \mathbf s]} \sE\Bigl( V^{(1)}(t_0) \Big| 
    \Delta\mathbf V^{(1)}=\mathbf z \Bigr)
    \sP(\Delta\mathbf V^{(1)}\in \rd\zbf).
    \label{mu_expl}
  \end{align}
Equation \eqref{mu_expl} shows that the convergence in \eqref{eq:01}
holds for all continuity points $\sbf\in [0, \infty)^{k}$ of the
  distribution function of $\Delta\Vbf^{(1)}$.  Since $s_0\geq 1$ was
  arbitrary, this concludes the proof.
\end{proof}

\begin{rem}
\begin{enumerate}
\item If $V^{(1)}(t_0)$ is stochastically independent of the
  increments $\Delta\mathbf V^{(1)}$, we simply have
  $\sP(\Delta\mathbf{\tilde{V}}^{(1)}\in d \mathbf z) =
  \sP(\Delta\mathbf{{V}}^{(1)}\in d \mathbf z)$.
\item If $p=\sP(V(t_0) = 0)=0$, the exponent measure $\mu$ of any
  finite-dimensional vector $\Xi=(\xi(t_0), \ldots, \xi(t_k))$, $t_0,
  \ldots, t_k\in T$, $k\in\Nb$, satisfies the condition
  $\mu\left( \{0\}\times [0,\infty)^k \right)=0,$  
  and following
    Proposition \ref{calculateW}, the incremental representation of
    $\Xi$ according to \eqref{def_xi} is given by $\Xi = \max_{i\in\Nb}
    U_i \cdot (1, \Delta\mathbf{\tilde{V}}_i)^\top$, where $\Delta\mathbf{\tilde{V}}_i$, $i\in\Nb$, are independent copies of $\Delta\mathbf{\tilde{V}}=\Delta\mathbf{\tilde{V}}^{(1)}$.
\item If $\xi$ admits a representation \eqref{def_xi}, we have
  $\sP(\Delta\mathbf{\tilde{V}}^{(1)}\in d \mathbf z) =
  \sP(\Delta\mathbf{{V}}\in d \mathbf z)$, which shows that 
\eqref{cond_incr_conv} is indeed 
a special case of Theorem \ref{theo_cond_increments_general}.
\end{enumerate}
\end{rem}

\begin{rem} \label{rem_thres}
  In the above theorem, the sequence $a(n)$ of thresholds is only
  assumed to converge to $\infty$, as $n\to\infty$, ensuring that
  $\{\eta(t_0) > a(n)\}$ becomes a rare event. For statistical applications
  $a(n)$ should also be chosen such that the number of exceedances
  \begin{align*}
    N(n) = \sum_{i=1}^n \1\{ \eta_i(t_0) > a(n) \}
  \end{align*}
  converges to $\infty$ almost surely, where $(\eta_i)_{i\in\Nb}$ is a
  sequence of independent copies of $\eta$. By the Poisson limit
  theorem, this is equivalent to the additional assumption that
  $\lim_{n\to\infty} a(n)/n = 0$, since in that case $n\sP(\eta(t_0) >
  a(n)) = n / a(n) \to \infty$, as $n\to\infty$.
\end{rem}
  
\begin{rem}
\cite{eng2012a} consider H\"usler-Reiss distributions
\cite{hue1989,kab2011} and obtain their limiting results by
conditioning on certain extremal events $A\subset E$.  They show that
various choices of $A$ are sensible in the H\"usler-Reiss case,
leading to different limiting distributions of the increments of
$\eta$. In case $\xi$ is a Brown-Resnick process and $A =
(1,\infty)\times [0, \infty)^{k}$ the assertions of Theorem
  \ref{theo_cond_increments_general} and \cite[Thm.\ 3.3]{eng2012a}
  coincide.
\end{rem}

\begin{ex}[Extremal Gaussian process \cite{sch2002}]
A commonly used class of stationary yet non-ergodic max-stable
processes on $\R^d$ is defined by
\begin{align}
  \label{schlather_model}
  \xi(t) = \max_{i\in\Nb} U_i Y_i(t), \quad t\in\R^d,
\end{align} 
where $\sum_{i\in\Nb} \delta_{U_i}$ is a Fr\'echet point process,
$Y_i(t)=\max(0, \tilde Y_i(t))$, $i \in \Nb$, and the $\tilde Y_i$ are
i.i.d.\ stationary, centered Gaussian processes with $\sE(\max(0,
\tilde Y_i(t))) =1$ for all $t\in\R^d$ \cite{sch2002,bla2011}.  Note
that in general, a $t_0\in\R^d$ s.t.\ $Y_i(t_0)=1$ a.s. does not
exist, i.e., the process admits representation \eqref{def_xi2} but not
representation \eqref{def_xi}.  In particular, for the extremal
Gaussian process we have $p=\sP(V(t_0)=0)=1/2$ and the distribution of
the increments in \eqref{density_increment} becomes
\begin{align*}
 \sP(\Delta\mathbf{\tilde{V}}^{(1)} \! \in \rd \mathbf z)
&= \frac12 \sE\Bigl[ Y(t_0) \, \Big|\,  
   (Y(t_1)/Y(t_0), \ldots, Y(t_k)/Y(t_0)) = \mathbf z, 
  \, Y(t_0)>0\Bigr]\\
  & \qquad \cdot\sP\Bigl( \bigl(Y(t_1)/Y(t_0), \ldots, 
  Y(t_k)/Y(t_0)\bigr) 
  \in \rd\zbf 
  \, \Big|\, Y(t_0)>0 \Bigr).
\end{align*}
\end{ex}

While the H\"usler-Reiss distribution is already given by the
incremental representation \eqref{def_xi}, cf.~\cite{kab2011}, other
distributions can be suitably rewritten, provided that the cumulative
distribution function and hence the respective exponent measure $\mu$
is known.

\begin{prop}\label{calculateW} 
  Let $\Xi = (\xi(t_0),\dots,\xi(t_k))$ be a max-stable
  process on $T = \{t_0, \ldots, t_k \}$ with standard Fr\'echet margins
  and suppose that its exponent measure $\mu$
  is concentrated on $(0, \infty) \times [0, \infty)^{k}$.  
  Define a random vector 
  $\Wbf=(W^{(1)}, \ldots, W^{(k)})$ 
  via its cumulative distribution function
  \begin{align}
    \label{def_W}
    \sP( \Wbf \leq \mathbf{s}) = \mu(B_{\mathbf{s}} \cap A), 
    \quad \mathbf{s}\in [0,\infty)^{k},
  \end{align}
  where $A = (1,\infty)\times [0, \infty)^{k}$ and $B_{\mathbf s} =
    \{\xbf \in [0,\infty)^{k+1}: \, (x^{(1)},\ldots,x^{(k)}) \leq
      x^{(0)} \sbf \}$.  Then, $\Xi$ allows for an incremental
      representation \eqref{def_xi} with $\Wbf_i$, $i\in\Nb$, being
      independent copies of $\Wbf$.
\end{prop}

\begin{proof}
 First, we note that \eqref{def_W} indeed defines a valid cumulative
 distribution function. To this end, consider the measurable
 transformation
 \begin{align*}
   T: (0, \infty)\times [0, \infty)^{k} 
     \to (0, \infty)\times [0, \infty)^{k}, \ (x_0,\dots, x_k)
   \mapsto \left(x_0, \frac{x_1}{x_0}, \dots, \frac{x_k}{x_0}\right).
 \end{align*}
Then, $ T(B_{\mathbf{s}} \cap A) = (1, \infty) \times [\0, \sbf]$ and
the measure $\mu^T(\cdot) = \mu(T^{-1}((1,\infty)\times \,\cdot\,))$
is a probability measure on $[0, \infty)^{k}$.
 Since
 \begin{align*}
   \mu(B_{\mathbf{s}} \cap A) 
   = \mu(T^{-1}((1,\infty)\times[\0, \sbf]))
   = \mu^T([\0, \sbf]),
 \end{align*}
the random vector $\Wbf$ is well-defined and has law $\mu^T$.

 By definition of the exponent measure, we have $\Xi \eqdist \max_{i
   \in \Nb} \Xbf_i$, where $\Pi = \sum_{i\in\Nb} \delta_{\Xbf_i}$ is a
 PPP on $E$ with intensity measure $\mu$. Then, the transformed point
 process $T\Pi = \sum_{i\in\Nb} \delta_{(X_i^{(0)},\,
   X_i^{(1)}/X_i^{(0)},\, \ldots,\, X_i^{(k)} /X_i^{(0)})}$ has
 intensity measure
 \begin{align*}
  \tilde \mu((c,\infty) \times [\0,\sbf]) 
  ={}& \mu\left(T^{-1}\left( (c,\infty) \times [\0,\sbf] \right) 
  \right)\\
  ={} & \mu(B_{\mathbf{s}} \cap ((c,\infty) \times [0,\infty)^k)) {}
  ={} c^{-1} \mu(B_{\mathbf{s}} \cap A)
 \end{align*}
 for any $c > 0$, $\mathbf{s} \in [0,\infty)^k$, where we use the fact
   that $\mu$, as an exponent measure, has the homogeneity property
   $c^{-1}\mu(\rd\xbf)=\mu(\rd(c\xbf))$.  Thus, $T\Pi$ has the same
   intensity as $\sum_{i\in\Nb} \delta_{(U_i, \Wbf_i)}$, where
   $\sum_{i\in\Nb} \delta_{U_i}$ is a Fr\'echet point process and
   $\Wbf_i$, $i \in \Nb$, are i.i.d.\ vectors with law $\sP(\Wbf \leq
   \mathbf{s}) = \mu(B_{\mathbf{s}} \cap A)$. Hence, we have
 \begin{align*}
  \Xi\eqdist{}& \max_{i \in \Nb} T^{-1}\left(\big(X_i^{(0)}, X_i^{(1)} / X_i^{(0)}, 
  \ldots, X_i^{(k)} / X_i^{(0)}\big)\right)\\
  \eqdist{}& \max_{i \in \Nb} T^{-1}\left(\big(U_i,\Wbf_i\big)\right) {}
  ={} \max_{i \in \Nb} U_i \Wbf_i,
 \end{align*}
 which completes the proof.
\end{proof}

\begin{ex}[Symmetric logistic distribution, cf.\ \cite{gum1960}] 
\label{ex:symm_log}
  For $T=\{t_0,\dots,t_k\}$, the symmetric logistic distribution is
  given by
  \begin{align}
    \sP(\xi(t_0) \leq x_0,\dots, \xi(t_k) \leq x_k) = 
    \exp\left[ - \left( x_0^{-q}+ \dots + x_k^{-q}\right)^{1/q} \right], 
    \label{eq:cdf_symm_log}
  \end{align}
  for $x_0,\dots,x_k>0$ and $q > 1$. Hence, the density of the
  exponent measure is
  \begin{align*}
    \mu(\rd x_0,\dots,\rd x_k) = 
    \left(\sum_{i=0}^k x_i^{-q}\right)^{1/q -(k+1)}
    \left(\prod_{i=1}^k(iq-1)\right)
    \prod_{i=0}^k x_i^{-q-1}  \rd x_0\dots \rd x_k.
  \end{align*}
  Applying Proposition \ref{calculateW}, the incremental process 
  $W$ in the representation \eqref{def_xi} is given by
  \begin{align*}
    \sP(W(t_1) \leq s_1, \dots W(t_k) \leq s_k) 
    = \left(1 + \sum_{i=1}^k s_i^{-q}\right)^{1/q - 1}.
  \end{align*}
\end{ex}

\subsection{Continuous sample paths}
In this subsection, we provide an analog result to Theorem
\ref{theo_cond_increments_general}, in which convergence in the sense
of finite-dimensional distributions is replaced by weak convergence on
function spaces. In the following, for a Borel set $U\subset\R^d$, we
denote by $C(U)$ and $C^+(U)$ the space of non-negative and strictly
positive continuous functions on $U$, respectively, equipped with the
topology of uniform convergence on compact sets.
  
\begin{thm} \label{theo_cond_increments_cont}
  Let $K$ be a compact subset of $\R^d$ and $\{\eta(t): \ t\in K\}$ be
  a process with positive and continuous sample paths in the MDA of a
  max-stable process $\{\xi(t): \ t\in K\}$ as in \eqref{def_xi} in
  the sense of weak convergence on $C(K)$. In particular, suppose that
  $$ \frac 1n \max_{i=1}^n \eta_i(\cdot) \cvgdist \xi(\cdot), \quad
  n\to\infty.$$ Let $W$ be the incremental process from \eqref{def_xi}
  and $Z$ a Pareto random variable, independent of $W$. Then, for any
  sequence $a(n)$ of real numbers with $a(n) \to \infty$, we have the
  weak convergence on $(0,\infty)\times C(K)$
  \begin{align*}
    \left(\frac{\eta(t_0)}{a(n)}, 
    \frac{\eta(\cdot)}{\eta(t_0)} \ \Big|\ \eta(t_0) > a(n) \right)
    \cvgdist (Z,  W(\cdot)),
  \end{align*}
  as $n$ tends to $\infty$.
\end{thm}

\begin{rem}\label{weak_conv_Rd}
Analogously to \cite[Thm.\ 5]{whi1970}, weak convergence of a sequence
of probability measures $P_n$, $n\in\Nb$, to some probability measure
$P$ on $C(\R^d)$ is equivalent to weak convergence of $P_n r_j^{-1}$
to $P r_j^{-1}$ on $C([-j, j]^d)$ for all $j\geq 1$, where $r_j :
C(\R^d) \to C([-j,j]^d)$ denotes the restriction of a function to the
cube $[-j, j]^d$.  Hence the assertion of Theorem
\ref{theo_cond_increments_cont} remains valid if the compact set $K$
is replaced by $\R^d$.
\end{rem}

\begin{proof}[Proof of Theorem \ref{theo_cond_increments_cont}]
 As the process $\xi$ is max-stable and $\eta\in\text{MDA}(\xi)$,
 similarly to the case of multivariate max-stable distributions
 (cf.\ Theorem \ref{theo_cond_increments_general}), we have that
 \begin{align} \label{conv_dehaan}
  \lim_{u \to \infty} u\sP(\eta / u \in B) = \mu(B)
 \end{align}
 for any Borel set $B \subset C(K)$ bounded away from $0^K$, i.e.,
 $\inf\{\sup_{s\in K} f(s) : \ f\in B\} > 0$, and with $\mu(\partial
 B) = 0$ \cite[Cor.\ 9.3.2]{deh2006a}, where $\mu$ is the
 \emph{exponent measure} of $\xi$, defined by
 \begin{align}
  & \sP(\xi(s) \leq x_j, \ s \in K_j, \ j=1,\ldots,m) \nonumber \\
  &={} \exp\left[-\mu\left(\left\{ f \in C(K): \ \textstyle\sup_{s \in K_j} f(s) > x_j \textrm{ for some } j \in \{1,\ldots,m\} 
  \right\}\right)\right]
 \end{align}
 for $x_j \geq 0$, $K_j \subset K$ compact.  Thus, $\mu$ equals the
 intensity measure of the Poisson point process $\sum_{i \in \Nb}
 \delta_{U_i W_i(\cdot)}$.  
For $z>0$ and $D\subset C(K)$ Borel, we consider the sets 
\begin{align*}
A_{z} &= \{f \in C(K): \ f(t_0) > z\}\\
B_D &= \{f \in C(K) : f(\cdot)/f(t_0)\in D\}
\end{align*}
and $A=A_1$. Note that $B_D$ is invariant w.r.t.\ multiplication by
 any positive constant. Then, as $W(t_0) = 1$ a.s., we have
 $\mu(A_{z}) = \int_{z}^\infty u^{-2} \sd u = z^{-1}$ and for
 $s_0\geq 1$ and any Borel set $D \subset C(K)$ with
 $\sP(W \in \partial D) = 0$, by \eqref{conv_dehaan}, we get
 \begin{align*}
  &\sP\left\{\eta(t_0) / a(n) > s_0,\ \eta(\cdot)/\eta(t_0) \in D \, \Big|\, \eta(t_0) > a(n) \right\}\\
  &= \frac{a(n) \sP\bigl\{\eta(\cdot) / a(n) \in A_{s_0} \cap B_D \cap A\bigr\}}{a(n) \sP\bigl\{\eta(\cdot) / a(n) \in A\bigr\}}\\
  &\stackrel{n \to \infty}{\longrightarrow}{} \frac{\mu(B_D \cap A_{s_0})}{\mu(A)}\\
 &={} \int_{s_0}^\infty u^{-2}\sP\bigl\{u W(\cdot) \in B_D\bigr\} \sd u\\
   &={} s_0^{-1} \sP\bigl\{W(\cdot) \in D\bigr\}, 
 \end{align*}
 which is the joint distribution of $Z$ and $W(\cdot)$.
\end{proof}

\begin{ex}[Brown-Resnick processes, cf.\ \cite{bro1977,kab2009}]
  \label{BRproc}
  For $T=\R^d$, $d\geq 1$, let $\{Y(t): \ t\in T\}$ be a centered
  Gaussian process with stationary increments, continuous sample paths
  and $Y(t_0) = 0$ for some $t_0\in\R^d$. Note that by
  \cite[Thm.\ 1.4.1]{adl2007} it is sufficient for the continuity of
  $Y$ that there exist constants $C,\alpha,\delta > 0$, such that
  \begin{align*}
    \sE |Y(s) - Y(t)|^2 \leq \frac{C}{|\log \|s-t\| |^{1+\alpha}}
  \end{align*}
  for all $s,t\in\R^d$ with $\|s-t\|<\delta$. Further let $\gamma(t) =
  \sE(Y(t) - Y(0))^2$ and $\sigma^2(t) = \sE(Y(t))^2$, $t \in \R^d$,
  denote the variogram and the variance of $Y$, respectively. Then,
  with a Fr\'echet point process $\sum_{i\in\Nb} \delta_{U_i}$ and
  independent copies $Y_i$ of $Y$, $i\in\Nb$, the process
  \begin{align}
    \label{BR_proc}
    \xi(t) = \max_{i\in\Nb} U_i \exp\left(Y_i(t) - \sigma^2(t) / 2\right), 
    \quad t\in\R^d,
  \end{align}  
  is stationary and its distribution only depends on the variogram
  $\gamma$.  Comparing \eqref{BR_proc} with the incremental
  representation \eqref{def_xi}, the distribution of the increments is
  given by the log-Gaussian random field $W(t) = \exp\left(Y(t) -
  \sigma^2(t) / 2\right)$, $t\in\R^d$, and Theorem
  \ref{theo_cond_increments_cont} applies.
\end{ex}

\section{Mixed moving maxima representation}\label{M3representation}

A large and commonly used class of max-stable processes is the
class of M3 processes \eqref{def_M3}. Let 
\begin{align}
   \label{pi0}
   \Pi_0 = \sum_{i\in\Nb} \delta_{(U_i.T_i,F_i)}
 \end{align}
 be the corresponding PPP on $(0,\infty)\times\R^d\times C(\R^d)$ with
 intensity $u^{-2}\rd u \,\rd t \,\sP_F(\rd f)$.  In the
 sequel, M3 processes are denoted by 
\begin{align*}
  M(t) = \max_{i\in\Nb} U_i F_i(t- T_i), \quad t\in\R^d.
\end{align*} 
The marginal distributions
 of $M$ are given by
\begin{align}
& \sP(M(t_0)\leq s_0, \ldots, M(t_k)\leq s_k) \notag\\
&= \sP\left[ \Pi_0 
  \left(\left\{(u,t,f): 
  \max_{l=0}^k u f(t_l-t)/s_l > 1\right\}\right) = 0\right] \notag\\
&= \exp\left(- \int_{C(\R^d)} \int_{\R^d} 
\max_{l=0}^k (f(t_l-t)/s_l)\, \rd t \, \sP_F(\rd f) \right),\label{M3_marginal}
\end{align}
$t_0, \ldots, t_k\in\R^d$, $s_0, \ldots, s_k\geq 0$, $k\in\Nb$.

\bigskip

In Section \ref{examples_increment_representation}, we were interested
in recovering the incremental process $W$ from processes in the MDA of
a max-stable process with incremental representation. In case of M3
processes, the object of interest is clearly the distribution of the
shape function $F$. Thus, in what follows, we provide the
corresponding convergence results for processes $\eta$ in the MDA of
an M3 process.  We distinguish between processes on $\R^d$ with
continuous sample paths and processes on a grid ($\Z^d$).  The main
idea is to consider $\eta$ in the neighborhood of its own (local)
maximum, conditional on this maximum being large.
\medskip

\subsection{Continuous Case}

Let $\{\eta(t): \, t \in \R^d\}$ be strictly positive and in the MDA
of a mixed moving maxima process $M$ in the sense of weak convergence
in $C(\R^d)$. We assume that $\eta$ is normalized such that the
norming functions in \eqref{MDA} are given by $c_n(t) = 1 / n$ and
$b_n(t) = 0$, for any $n\in\Nb$ and $t\in\R^d$. Further suppose that
the shape function $F$ of $M$ is sample-continuous and satisfies
\begin{align}
  \begin{split}
    F(\vec 0) &= \lambda \quad a.s., \\
    F(t) & \in [0,\lambda) \ \forall t \in \R^d \setminus \{\vec 0\} \quad a.s.
      \label{eq:Fmaxatorigin}
  \end{split}
\end{align}
for some $\lambda > 0$ and 
\begin{equation}
  \int_{\R^d} \sE\left\{  \max_{t_0 \in K} F(t_0 - t) \right\} \sd t < \infty 
  \label{eq:sup-integrability-cont}
\end{equation}
for any compact set $K \subset \R^d$.  Under these assumptions, there
is an analog result to Theorem \ref{theo_cond_increments_cont}.

\begin{thm} \label{thm_conv_mmm}
 Let $\Q, K \subset \R^d$ be compact such that $\partial \Q$ is a
 Lebesgue null set and let
 $$\tau_\Q: \ C(\Q) \to \R^d, \ f \mapsto \inf\left( \argmax_{t \in \Q}
 f(t) \right),$$
 where {\rm ``inf''} is understood in the lexicographic sense.
 Then, under the above assumptions, for any Borel set
 $B \subset C(K)$ with $\sP(F / \lambda \in \partial B) = 0$, and any
 sequence $a(n)$ with $a(n) \to \infty$ as $n \to \infty$, we have
 \begin{align*}
  & \lim_{\substack{\{\vec 0\} \in L \nearrow \R^d\\ {\rm compact}}}
   \limsup_{n \to \infty} \sP\Big\{ \eta\big(\tau_\Q(\eta|_\Q)+\cdot\big) \big/  \eta(\tau_\Q(\eta|_\Q)) \in B \ \Big| \\[-1em]
   &  \hspace{2.5cm} \max_{t\in \Q}\eta(t) = \max_{t \in \Q \oplus L} \eta(t), 
   \ \max_{t\in \Q}\eta(t) \geq a(n)\Big\} \hfill {}={} \hfill \sP\big\{F(\cdot) / \lambda \in B\big\},
 \end{align*}
 where $\oplus$ denotes morphological dilation.
 
 The same result holds true if we replace $\limsup_{n \to \infty}$ by
 $\liminf_{n \to \infty}$.
\end{thm}

\begin{proof}
First, we consider a fixed compact set $L\subset\R^d$ 
large enough such that $K \cup \{{\bf 0}\} \subset L$ and define 
\begin{align*}
  A_L = \left\{ f \in C(\Q\oplus L): \ \max_{t\in \Q}f(t) \geq 1, \ \max_{t\in \Q}f(t) = \max_{t \in \Q \oplus L} f(t)\right\}
 \end{align*}
 and 
\begin{align*}
  C_B = \left\{f \in C(\Q \oplus L):
\ f\big(\tau_\Q(f|_\Q) + \,\cdot\,\big) \big/ f(\tau_\Q(f|_\Q)) \in B\right\}
 \end{align*}
 for any Borel set $B \subset C(K)$. Note that $C_B$ is invariant w.r.t.\ multiplication by any positive constant.
 Thus, we get
 \begin{align}
  & \sP\Big\{\eta\big(\tau_\Q(\eta|_\Q) + \cdot\big) 
   \big/ \eta(\tau_\Q(\eta|_\Q)) \in B \ \Big|\  \max_{t\in \Q}\eta(t) = \max_{t \in \Q \oplus L} \eta(t)  \geq a(n)\Big\} 
   \nonumber \\
  & ={} \sP\big\{\eta / a(n) \in C_B \,\big|\, \eta / a(n) \in A_L\big\} 
  \nonumber\\
  & ={} \frac{a(n) \sP\big\{\eta/a(n) \in C_B,\,\eta/a(n) 
    \in A_L \big\}}{a(n) \sP\big\{\eta/a(n) \in A_L \big\}}.
  \label{eq:expand-cont}
 \end{align}
 
  By \cite[Cor.\ 9.3.2]{deh2006}  and \cite[Prop.\ 3.12]{res2008} we have 
 \begin{align*}
  \limsup_{u \to \infty} u\sP(\eta / u \in C) \leq{}& \mu(C), \quad C \subset C(\Q \oplus L) \text{ closed},\\
  \liminf_{u \to \infty} u\sP(\eta / u \in O) \geq{}& \mu(O), \quad O \subset C(\Q \oplus L) \text{ open},
 \end{align*}
  where $C$ and $O$ are bounded away from $0^K$.
  Here, $\mu$ is the intensity measure of the PPP $\sum_{i \in \Nb} \delta_{U_i  F_i(\,\cdot\, -
  T_i)}$ restricted to $C(\Q \oplus L)$.
  Thus, by adding or removing the boundary, we see that all the limit points of Equation \eqref{eq:expand-cont} lie in
  the interval
  \begin{equation} \label{eq:liminterval}
   \left[ \frac{\mu(C_B \cap A_L) - \mu(\partial (C_B \cap A_L))}{\mu(A_L) + \mu(\partial A_L)}, \frac{\mu(C_B \cap A_L) + \mu(\partial (C_B \cap A_L))}{\mu(A_L) - \mu(\partial A_L)}\right].
  \end{equation}
  We note that $A_L$ is closed and the set
  \begin{align*}
  A_L^* ={} & \bigg\{ f \in C(\Q \oplus L): \\[-.5em]
  &\quad \, \tau_\Q(f|_\Q) \in \Q^o, \ \max_{t\in \Q} f(t) > \max\big\{1, f(t)\big\} \ \forall t \in \Q \oplus L \setminus\{\tau_\Q(f|_\Q)\} \bigg\}
  \end{align*}
  is in the interior of $A_L$ (Lemma \ref{lem:AL}). Hence, we can assess
  \begin{align}
   \mu(\partial A_L) \leq{} & \quad \ \mu(\{f \in C(\Q \oplus L): \ \max_{t\in \Q}f(t) = 1\}) \notag\\
                            &      + \mu\bigg( \quad \bigg( \quad \{f \in C(\Q \oplus L): \ \tau_\Q(f|_\Q) \in \partial \Q\} \notag\\
                            & \hspace{1.55cm}  \cup \left\{f \in C(\Q \oplus L): \ \argmax_{t \in \Q\oplus L} f(t) \text{ is not unique}\right\}\bigg) \notag\\
                            & \qquad \cap \left\{f \in C(\Q \oplus L): \ \max_{t\in \Q}f(t) = \max_{t \in \Q \oplus L} f(t) \geq 1 \right\}\bigg) \notag \\
                     \leq{} & 0 + \int_{\partial \Q} \int_{\lambda^{-1}}^\infty u^{-2} \sd u \sd t_0 \notag\\
                            & \phantom{0} + \int_{\R^d \setminus (\Q \oplus L)} \int_{\lambda^{-1}}^\infty u^{-2} \sP\left\{u  \max_{t_0 \in \Q} F(t_0 -x) \geq 1\right\} \sd u \sd x \label{eq:partAL}.
  \end{align}
  Here, the equality
  $\mu(\{f \in C(\Q \oplus L): \ \max_{t\in \Q}f(t) = 1\}) = 0$ holds
 as $\max_{t \in \Q} M(t)$ is Fr\'echet distributed
  (cf.~\cite[Lemma 9.3.4]{deh2006}). Since $\partial \Q$ is a Lebesgue
  null set, the second term on the right-hand side of
  \eqref{eq:partAL} also vanishes. Thus, 
  \begin{align}
   \mu(\partial A_L) \leq{} & \int_{\R^d \setminus (\Q \oplus L)} \int_{\lambda^{-1}}^\infty u^{-2} \sP\left\{u  \max_{t_0 \in \Q} F(t_0 -x) \geq 1\right\} \sd u \sd x =: c(L) \label{eq:partAL2}.
  \end{align} 
  
  Now, let $B \subset C(K)$ a be Borel set such that $\sP(F / \lambda \in \partial B) = 0$. For the set $C_B$, we obtain that the set
  \begin{align*}
     C_B^* ={} & \bigg\{ f \in C(\Q \oplus L): \ \argmax_{f \in \Q} f(t) \text{ is unique},\  \frac{f\big(\tau_\Q(f|_\Q) + \cdot\big)}{f(\tau_\Q(f|_\Q))} \in B^o \bigg\}
  \end{align*}
  is in the interior of $C_B$ and that the closure of $C_B$ is a subset of
  \begin{align*}
   C_B^* \cup{} & \left\{f \in C(\Q \oplus L): \, \argmax_{t \in \Q} f(t) \text{ is not unique}\right\}\\
     \cup{} & \left\{f \in C(\Q \oplus L): \, f\big(\tau_\Q(f|_\Q) + \cdot\big) \big/ f(\tau_\Q(f|_\Q)) \in \partial B\right\}
  \end{align*}
  (Lemma \ref{lem:interCB} and Lemma \ref{lem:CB}). 
  Thus, by \eqref{eq:partAL2}, we can assess
  \begin{align}
   \mu(\partial (C_B \cap A_L)) \leq{} & \mu(\partial A_L) + \mu(\partial C_B \cap A_L) \notag\\
   \leq{} & c(L) + \int_{\R^d \setminus (\Q \oplus L)} \int_{\lambda^{-1}}^\infty u^{-2} 
   \sP\left\{u  \max_{t_0 \in \Q} F(t_0 -x) \geq 1\right\}\sd u \sd x \notag\\
          & \hspace{0.7cm}  + \int_{\Q} \int_{\lambda^{-1}}^\infty u^{-2} \sP(F / \lambda \in \partial B) \sd u \sd t \quad
       {}={} \quad 2 c(L). \label{eq:partCB}
  \end{align}

  Furthermore, we get
 \begin{align}
   & \mu(C_B \cap A_L) \nonumber \\
   ={} & \int_\Q \int_{\lambda^{-1}}^\infty u^{-2} \sP\Big\{F(\cdot) / \lambda \in B\Big\} 
   \sd u \sd t_0 \nonumber \\
   &  + \int_{\R^d \setminus(\Q \oplus L)} \int_{\lambda^{-1}}^\infty u^{-2} 
   \sP\bigg\{u  \max_{t_0 \in \Q} F(t_0 -x) \geq 1,\ \nonumber \\
   &     \hspace{2.5cm} 
   F\left(\Big(\tau_\Q(F(\cdot-x)|_\Q)\Big)+\cdot-x\right) \Big/ \max_{t_0 \in \Q} F(t_0-x) \in B,\nonumber \\
   &     \hspace{2.5cm} F(t-x) / \max_{t_0 \in \Q} F(t_0-x) \leq 1 
   \ \forall t \in \Q \oplus L \bigg\} \sd u \sd x. \label{eq:CB}
 \end{align}
 The second term in \eqref{eq:CB} is positive and can be bounded from above by $c(L)$.
 Setting $B= C(K)$, $\mu(A_L)$ can be expressed in an analogous way.  
 Now, we plug in the results of \eqref{eq:partAL2}, \eqref{eq:partCB} and \eqref{eq:CB} into \eqref{eq:liminterval} to obtain that all the limit points of \eqref{eq:expand-cont}
 are in the interval
 \begin{align*}
  \left[\frac{\lambda \cdot |\Q| \cdot \sP\big\{F(\cdot) / \lambda \in B\big\} - 2c(L)}{\lambda \cdot |\Q| + 2c(L)},
        \frac{\lambda \cdot |\Q| \cdot \sP\big\{F(\cdot) / \lambda \in B\big\} + 3c(L)}{\lambda \cdot |\Q| - c(L)} \right].
 \end{align*}

 Finally, we note that $c(L)$ can be bounded from above by
 $$\int_{\R^d \setminus (\Q \oplus L)} \sE \Big\{\max_{t_0 \in \Q}
 F(t_0-x)\Big\} \sd x,$$ which vanishes for $L \nearrow \R^d$ because
 of assumption \eqref{eq:sup-integrability-cont}.  This yields the
 assertion of the theorem.
  \end{proof}
\medskip

We conclude the treatment of the continuous case with an example of a
process $\eta$ that allows for an application of Theorem
\ref{thm_conv_mmm}. As $\eta$ will be composed of a (locally) finite
number of shape functions from the M3 construction in \eqref{def_M3},
$\eta$ may directly model rainfall data and has therefore the
potential for various practical applications.

\begin{ex} \label{ex:mmm-mda}
Let $\{F(t): \ t \in \R^d\}$ be a random shape function as defined 
in \eqref{assumption_integral}.  For $c, \epsilon >
0$ let $\Pi_{c, \epsilon} = \sum_{i\in\Nb} \delta_{(U_i.T_i,F_i)}$ be a
PPP on $(0,\infty)\times\R^d\times C(\R^d)$ with
intensity
\begin{align*}
c\1_{\{u\geq\epsilon\}}u^{-2}\rd u\,\rd t\,\sP_F(\rd f).
\end{align*}
and, for $\kappa > 0$, define a process $\tilde M=\tilde M_{c,\epsilon,\kappa}$ 
by 
$$\tilde M(\cdot) = \kappa \vee \max_{(u, t, f)\in \Pi_{c, \epsilon}}
u f(\,\cdot\, - t).$$  
Then, the following statements hold.
\begin{enumerate}
 \item[1. ] $\tilde M$ is in the MDA of the M3 process $M$ associated to $F$
 in the sense of finite-dimensional distributions. 
 \item[2. ] If $F$ satisfies \eqref{eq:sup-integrability-cont},
 then $\tilde M$ is in the MDA of $M$ in the sense of weak convergence
 on $C(\R^d)$.
\end{enumerate}
\end{ex}
For a proof of this example, the reader is referred to Appendix
\ref{sec:proof_ex_mmm-mda}.

\subsection{Discrete Case}

Theorem \ref{thm_conv_mmm} allows for estimation of $F$ if the
complete sample paths of $\eta$ are known, at least on a large set $\Q
\oplus L \subset \R^d$.  For many applications, this assumption might
be too restrictive. Therefore, we seek after a weaker assumption that
only requires to know $\eta$ on a grid. This needs a modification of
the underlying model leading to a discretized mixed moving maxima
process.

Let $\{F(t): \ t \in \Z^d\}$ be a measurable stochastic process with
values in $[0, \infty)$ and
\begin{align}
  \label{assumption_integral_discr}
  \sum_{t \in \Z^d} \sE F(t)  = 1.
\end{align}
Further, let $\Pi_{0,\discr} = \sum_{i\in\Nb} \delta_{(U_i.T_i,F_i)}$
be a Poisson point process on $(0,\infty) \times \Z^d\times
[0,\infty)^{\Z^d}$ with intensity $u^{-2}\rd u\, \delta_{\Z^d}(\rd t)
  \, \sP_F(\rd f)$.  Then, the discrete mixed moving maxima process
  $M_{\discr}$ is defined by
\begin{align}
  M_{\discr}(t) = \max_{i\in\Nb} U_i F_i(t- T_i), \quad t\in\Z^d. 
  \label{def_M3_discr}
\end{align}
The process $M_{\discr}$ is max-stable and stationary on $\Z^d$ and
has standard Fr\'echet margins.
\medskip

Let $\{\eta(t), \ t \in \Z^d\}$ be in the MDA of a discrete mixed
moving maxima process $M_{\discr}$ in the sense of convergence of
finite-dimensional distributions with norming functions $c_n(t) = 1/n$
and $b_n(t) = 0$ in \eqref{MDA}, $n\in\Nb$ and
$t\in\Z^d$. Furthermore, we assume that the shape function $F$
satisfies \eqref{eq:Fmaxatorigin} with $\R^d$ being replaced by
$\Z^d$.  Then, analogously to Theorem \ref{thm_conv_mmm}, the
following convergence result can be shown.

\begin{thm}
 Under the above assumptions, for any $k \in \Nb$, $k+1$ distinct
 points $t_0, \ldots, t_k \in \Z^d$, any Borel sets $B_1, \ldots, B_k
 \subset [0,\infty)$ such that
 $$ \sP\big\{(F(t_1)/\lambda,\ldots,F(t_k)/\lambda) \in
   \partial(B_1\times\cdots\times B_k)\big\}=0,$$ and any sequence
   $a(n)$ with $a(n) \to \infty$ as $n \to \infty$, it holds
 \begin{align*}
   \lim_{\substack{\{\vec 0\} \in L \nearrow \Z^d\\ {\rm compact}}}
          \lim_{n \to \infty} &\sP\big\{ \eta(t_0+t_i) / \eta(t_0) \in B_i, \ i=1,\ldots,k \ \big| \\[-1em]
               & \hspace{3.5cm} \eta(t_0) = \max_{t \in L} \eta(t_0+t),\ \eta(t_0) \geq a(n)\big\}\\[.5em]
          = &\sP\big\{F(t_i) / \lambda \in B_i, \ i=1,\ldots,k\big\}.
 \end{align*}
\end{thm}

\bigskip

\section{Switching between the different representations}\label{sec:switching}

In the previous sections we analyzed processes that admit the 
incremental representations \eqref{def_xi} or \eqref{def_xi2} and, 
on the other hand, processes of M3 type as in \eqref{def_M3}.
We show that under certain assumptions, we can switch 
from one representation to the other.

\subsection{Incremental representation of mixed moving maxima processes}

We distinguish between M3 processes with strictly positive shape
functions, for which we can find an incremental representation \eqref{def_xi},
and general non-negative shape functions, for which only the weaker representation
\eqref{def_xi2} can be obtained.

\subsubsection{Mixed moving maxima processes with positive shape functions}\label{ex:M3}

\begin{thm}
Let $M$ be an M3 process on $\R^d$ as in \eqref{def_M3} with a
shape function $F$ with $F(t) > 0$ for all $t \in \R^d$. 
Then $M$ admits a representation \eqref{def_xi} with $t_0 = 0$ and incremental 
process $W$ given by 
\begin{align}
  \sP(W\in L) 
  = \int_{C^+(\R^d)} \int_{\R^d} \1_{\{f(\cdot - t)/f(-t)\in L\}} f(-t)
  \sd t \, \sP_F(\rd f), \quad L\in\mathcal B(C^+(\R^d)).\label{defWofM3}
\end{align}
\end{thm}
\begin{proof}
We consider the two Poisson point processes on $(0,\infty)\times C^+(\R^d)$
\begin{align}
  \Pi_1 = \sum_{i\in\Nb}\delta_{(U_i F_i(-T_i), F_i(\cdot - T_i)/F_i(-T_i))}, \label{auxPPP1}
\end{align}
as a transformation of $\Pi_0$ in \eqref{pi0}, and
\begin{align}
  \Pi_2 = \sum_{i\in\Nb}\delta_{(U'_i, W_i(\cdot))}, \label{auxPPP2}
\end{align}
with $W_i$, $i\in\Nb$, being independent copies of $W$, and with $\sum_{i\in\Nb}\delta_{U'_i}$ being
a Fr\'echet point process.
Then the intensity measures of $\Pi_1$ and $\Pi_2$ satisfy
\begin{align*}
  &\sE\Pi_1([z,\infty)\times L)\\
    &=\int_{C^+(\R^d)}\int_{\R^d}\int_0^\infty u^{-2} \1_{\{u f(-t) \geq z\}}
    \1_{\{f(\cdot - t)/f(-t)\in L\}}
    \, \rd u \, \rd t \, \sP_F(\rd f)\\
    &= z^{-1} \int_{C^+(\R^d)}\int_{\R^d} \1_{\{f(\cdot - t)/f(-t)\in L\}} f(-t)
    \, \rd t \, \sP_F(\rd f)\\
    &=z^{-1} \sP(W\in L)\\
    &=\sE\Pi_2([z,\infty)\times L), 
\end{align*}
$L\in\mathcal B({C^+(\R^d)})$, $z> 0$, and hence $\Pi_1\eqdist\Pi_2$.
The assertion follows from the fact that $M$ is uniquely determined by
$\Pi_1$ via the relation $M(t) = \max_{(v,g)\in\Pi_1} v g(t)$,
$t\in\R^d$.
  \end{proof}

While the definition of $W$ in \eqref{defWofM3} is rather implicit,
in the following, we provide an explicit construction of the
incremental process $W$, which can also be used for simulation.
To this end, let $\sum_{i\in\Nb}\delta_{U_i''}$ be a Fr\'echet point process
and let the distribution of 
$(S,G)\in C^+(\R^d)\times\R^d$ be given by
\begin{align}
  \label{hat_distr}
  &\sP\bigl((S, G)\in (B\times L)\bigr)\\ &= 
  \int_{C^+(\R^d)}\int_{\R^d} \1_{s\in B}\1_{f\in L}
  \frac{f(-s)}{\int f(r) \rd r} \,\rd s \left(\int f(r) \rd r\right)
  \sP_F(\rd f) \notag\\
  \notag &= \int_{C^+(\R^d)}\int_{\R^d} \1_{s\in B}\1_{f\in L} f(-s) \, \rd s \, \sP_F(\rd f),
\end{align}
$B\in\mathcal B^d$, $L\in\mathcal B({C^+(\R^d)})$.  In other words,
$\sP_G(\rd f)= (\int f(r)\sd r)\,\sP_F(\rd f)$ and, conditional on
$\{G=f\}$, the density function of the shift $S$ is proportional to
$f(- \cdot)$.  Putting $W(\cdot) = G(\cdot - S)/G(-S)$, equation
\eqref{defWofM3} is satisfied and with i.i.d.\ copies $W_i$,
$i\in\Nb$, of $W$, we get that $\max_{i\in\Nb} U_i'' W_i(\cdot)$ is
indeed an incremental representation \eqref{def_xi} of the mixed
moving maxima process $M$.

\begin{rem}[M3 representation of Brown-Resnick processes, cf. \cite{kab2009}]
\label{mmm_BRproc}
We consider the following two special cases of mixed moving maxima
processes:
\begin{enumerate}
  \item
  Let $\Sigma\in\R^{d\times d}$ be a positive definite matrix and let
  the shape function be given by $F(t) = (2\pi)^{-d/2} |\Sigma|^{-1/2}
  \exp\left\{-\frac{1}{2}t^\top\Sigma^{-1} t\right\}$, $t\in\R^d$.
  Then, $M$ becomes the well-known Smith process. At the same time, by
  \eqref{hat_distr}, $S\sim N(0,\Sigma)$ and $G \equiv F$. Thus
  \begin{align*}
    Y(t)&=\exp\left\{-\textstyle\frac{1}{2}(t-S)^\top\Sigma^{-1}(t-S) + \frac{1}{2}S^\top\Sigma^{-1}S\right\}\\
    &=  \exp\left\{-\textstyle\frac{1}{2}t^\top\Sigma^{-1}t + t^\top \Sigma^{-1} S\right\}. 
  \end{align*}
  Since $\sE(t^\top \Sigma^{-1} S)^2 = t^\top\Sigma^{-1}t $, $M$ is
  equivalent to the Brown-Resnick process in \eqref{BR_proc} with
  variogram $\gamma(h) = h^\top \Sigma^{-1}h$.
  \item
  For the one-dimensional Brown-Resnick process $\xi$ in
  \eqref{BR_proc} with variogram $\gamma(h) =|h|$, i.e., $Y$ is the
  exponential of a standard Brownian motion with drift $-|t| / 2$,
  \cite{eng2011} recently showed that the M3 representation is given
  by $\{F(t): \, t\in\R\} = \{ Y(t) \mid Y(s)\leq 0 \ \forall s \in \R:
  \, t\in\R\}$, i.e., the shape function is the exponential of a conditionally 
  negative drifted Brownian motion.  
  Having these two representations, it follows that 
  the law of the conditional Brownian motion $F$, re-weighted by
  $\int F(t) \rd t$ and randomly shifted with density
  $F(-\cdot) / \int F(t) \rd t$, coincides with the law of $Y$.
\end{enumerate}
\end{rem}

\subsubsection{Mixed moving maxima processes with finitely supported shape functions\newline}\label{ex:M32}%
\ Let $M$ be an M3 process on $\R^d$ as in \eqref{def_M3}.  In
contrast to Section \ref{ex:M3}, where the shape functions are
required to take positive values, here, we allow for arbitrary shape
functions with values in $[0, \infty)$.

\begin{thm}
\label{finite_supp}
The M3 process $M$ as in \eqref{def_M3} allows for an incremental
representation of the form \eqref{def_xi2}, with incremental processes
$V_i$ given by
$$V_i(\cdot) = F_i(\cdot -R_i) / g(R_i).$$
Here $R_i$, $i\in\Nb$, are i.i.d.\ copies of a random vector $R$ with
arbitrary density $g$ satisfying $g(t)>0$ for all $t\in\R^d$, and
$F_i$, $i\in\Nb$, are i.i.d.\ copies of the random shape function
$F$.
\end{thm}
\begin{proof}
With $\sum_{i\in\Nb} \delta_{U_i}$ being a Fr\'echet point process,
we consider the process
\begin{align*}
\tilde M(t) = \max_{i\in\Nb} U_i F_i(t-R_i) / g(R_i), \qquad t\in\R^d,
\end{align*}
which clearly is of the form \eqref{def_xi2}.
Then,
\begin{align*}
  \sP&(\tilde M(t_0)\leq s_0, \ldots, \tilde M(t_k)\leq s_k) \\ 
  &= \exp\left(- \int_{C(\R^d)}\int_{\R^d} 
  \max_{l=0}^k (f(t_l-t)/(g(t)s_l)) g(t) \, \rd t \, \sP_F(\rd f)\right)\\
  &= \exp\left(- \int_{C(\R^d)}\int_{\R^d} 
  \max_{l=0}^k (f(t_l-t)/s_l)) \, \rd t \, \sP_F(\rd f)\right).
\end{align*}
The right-hand side coincides with the marginal distribution of $M$,
which is given by \eqref{M3_marginal}.  This concludes the proof.
  \end{proof}

Decomposing $V$ as in \eqref{decomp_V} with $t_0=0$, 
we obtain the equality in distribution
\begin{align*}
  V^{(1)}(\cdot) \eqdist \bigl(F(\cdot -R) / g(R) \ \big|  
    -R\in \supp(F) \bigr).
\end{align*}
Applying Theorem \ref{theo_cond_increments_general} yields
\begin{align}
&\sP\bigl(\Delta\mathbf{\tilde{V}}^{(1)}\in \rd\zbf\bigr)\notag\\
&= \sP\bigl( F(-R) / g(R) > 0\bigr) 
  \cdot\int_0^\infty y 
  \sP\big(V^{(1)}(0)\in \rd y,\ \Delta\Vbf^{(1)}\in\rd\zbf\big)\notag\\
&= \int_{C(\R^d)} \int_{-\supp(f)} g(s)\sd s \,\sP_F(\rd f)\notag\\
&\hspace{1cm}
  \cdot\int_0^\infty y \int_{C(\R^d)} \int_{-\supp(f)} \1_{f(-t)/g(t)\in \rd y} \1_{(f(t_l-t)/f(-t))_{l=1}^k \in \rd\zbf} \notag\\
&\hspace{6cm} \cdot g(t) 
        \left(\textstyle\int_{-\supp(f)} g(s)\rd s\right)^{-1} \sd t \,\sP_F(\rd f) \sd y\notag\\
&= \int_{C(\R^d)} \int_{\supp(f)} g(-s)\sd s \sP_F(\rd f)\notag\\
&\hspace{1cm}
  \cdot \int_{C(\R^d)} \int_{\supp(f)} f(t) \1_{(f(t_l+t)/f(t))_{l=1}^k \in \rd\zbf}  \left(\textstyle\int_{\supp(f)} g(-s)\rd s\right)^{-1} \sd t \,\sP_F(\rd f).
\label{distr_gen_incr}
\end{align}
If the shape function $F$ is deterministic, the right-hand side of 
\eqref{distr_gen_incr} simplifies to 
$ \int_{\supp(f)} f(t) \1_{(f(t_l+t)/f(t))_{l=1}^k \in \rd\zbf} \sd t$, i.e.,
the asymptotic conditional increments of
$\eta\in\text{MDA}(M)$ can be seen as a convolution of the shape
function's increments with a random shift, whose density is given by
the shape function itself. Note in particular, that this distribution
is independent of the choice of the density $g$ in Theorem
\ref{finite_supp}.

\begin{rem}
Section \ref{ex:M3} considers the subclass of M3 processes with
strictly positive shape functions and provides an incremental
representation as in \eqref{def_xi}, which is nicely related to the
conditional increments of $\eta$ due to the property $W(0)=1$. Section
\ref{ex:M32} applies to arbitrary M3 processes but only yields an
incremental representation as in \eqref{def_xi2}, for which the 
incremental process $V$ does not directly represent the
conditional increments of $\eta$.
\end{rem}

\subsection{Mixed moving maxima representation of the incremental construction}

\begin{thm}\label{theo:M3ofIncremental}
  Let $\sum_{i \in \Nb} \delta_{U_i}$ be a Fr\'echet point process and 
  let $W_i$, $i \in \Nb$, be
  independent copies of a non-negative, sample-continuous process $\{W(t),
  \ t\in\R^d\}$, satisfying
  \begin{align*}
    \lim_{||t|| \to \infty} W(t) &= 0 && a.s.,\\
    \sE W(t) &=1 &&\text{for all } t \in \R^d,\\
    \text{and} \quad  \sE\left\{\textstyle\max_{t \in K} W(t)\right\} & < \infty 
    &&\text{for any compact set $K \subset \R^d$.}
  \end{align*}
  Furthermore, let $W$ be Brown-Resnick stationary, i.e., the process $\xi$, defined by
  $$ \xi(t) = \max_{i \in \Nb} U_i W_i(t), \quad t \in \R^d,$$
  is stationary with standard Fr\'echet margins.
  Then, the following assertions hold:
  \begin{enumerate}
   \item The random variables
   \begin{align*}
     \tau_i = \inf \left\{\argsup_{t \in \R^d} W_i(t)\right\} \quad \text{and} 
     \quad \gamma_i = \sup_{t \in \R^d} W_i(t)
   \end{align*}
   are well-defined. 
   Furthermore,
   $\sum_{i \in \Nb} \delta_{(U_i \gamma_i, \tau_i, W_i(\cdot + \tau_i)
     / \gamma_i)}$ is a Poisson point process on $(0,\infty) \times
   \R^d \times C(\R^d)$ with intensity measure $ \Psi(\rd u, \sd t,
   \sd f) = c u^{-2} \sd u \, \sd t \, \sP_{\tilde F}(\rd f)$ for some $c >
   0$ and some probability measure $\sP_{\tilde F}$.
   \item $\xi$ has an M3 representation 
   with $\sP_F(\rd f) = \sP_{\tilde F}(c\sd f)$
   being the probability measure of the shape function $F$.
   The constant $c>0$ is given by
   \begin{align} \label{eq:c}
     c = \left(\int_{\R^d} \int_{C(\R^d)} f(t) \, \sP_{\tilde F}(\rd f) \sd t\right)^{-1}
   \end{align}
   and the probability measure $\sP_{\tilde F}$ is defined by
   $$\sP_{\tilde F}(A) = \frac{\int_0^\infty y \sP(W(\cdot +\tau) / y \in A, \ \tau \in
     K \mid \gamma=y) \, \sP_\gamma(\rd y)} {\int_0^\infty y \sP(\tau \in K
     \mid \gamma=y) \, \sP_\gamma(\rd y)}$$ for any Borel set $A \subset C(\R^d)$
     and any  compact set $K \subset \R^d$, where $\tau$ and $\gamma$ are defined as $\tau_i$ and
   $\gamma_i$, respectively, replacing $W_i$ by $W$, and $\sP_\gamma$
   is the probability measure belonging to $\gamma$.
  \end{enumerate}
\end{thm}

\begin{proof}
 \begin{enumerate}
  \item Analogously to the proof of \cite[Thm.\ 14]{kab2009}.
  \item From the first part it follows that
        $$\Phi_0 = \sum_{i \in \Nb} \delta_{(U_i \gamma_i / c,\, \tau_i,\,
    c \cdot W_i(\cdot + \tau_i) / \gamma_i)}$$ is a PPP with intensity measure
    $\Psi_0(\rd u, \sd t, \sd f) = u^{-2} \sd u \times \sd t \times \, \sP_F(\rd f)$ where 
    $\sP_F(\rd f) = \sP_{\tilde F}(c \sd f)$.  Hence, $\Phi_0$ is of the same type
    as $\Pi_0$ from the beginning of Section \ref{M3representation} and
        $$ \xi(t) = \max_{(y,s,f) \in \Phi_0} y f(\cdot - s), \quad
    t \in \R^d,$$ is a mixed moving maxima representation.  The
    integrability condition \eqref{assumption_integral} follows from
    the fact that $\xi$ has standard Fr\'echet marginals.  Thus,
        $$ \int_{\R^d} \int_{C(\R^d)} c f(t) \, \sP_{\tilde F}(\rd f) \sd t= 1,$$ which
    implies \eqref{eq:c}.
    In order to calculate $\sP_{\tilde F}$, let $A \in \mathcal B (C(\R^d))$ and
    $K \in \mathcal{B}^d$ be compact. The first part of this Theorem implies that
        $$ \Psi([1, \infty)\times K \times A) = c \cdot |K| \cdot \sP_{\tilde F}(A). $$
        Therefore,
      \begin{align} \label{eq:qlim}
         \sP_{\tilde F}(A) = \frac{ \Psi([1, \infty) \times K \times A)}{ 
             \Psi([1, \infty) \times K \times C(\R^d))},
        \end{align}
        and both the enumerator and the denominator are finite.  For
        the enumerator, we get
        \begin{align*}
              \Psi(&[1, \infty) \times K \times  A)\\
         ={} & \int_0^\infty u^{-2} \int^{\infty}_{u^{-1}} \sP(W(\cdot + \tau) / \gamma \in A, \ \tau \in K \mid \gamma=y) \, \sP_\gamma(\rd y) \sd u\\
         ={} & \int_0^\infty \int^{\infty}_{y^{-1}} u^{-2} \sd u  \cdot \sP(W(\cdot + \tau) / y \in A, \ \tau \in K \mid \gamma=y) \, \sP_\gamma(\rd y)\\
         ={} & \int_0^\infty y \sP(W(\cdot + \tau) / y \in A, \ \tau \in K \mid \gamma=y) \, \sP_\gamma(\rd y).
        \end{align*}
        Thus, by \eqref{eq:qlim},
        \begin{align*}
         \sP_{\tilde F}(A) = \frac{\int_0^\infty y \sP(W(\cdot + \tau) / y \in A, \ \tau \in K \mid \gamma=y) \, \sP_\gamma(\rd y)}{\int_0^\infty y \sP(\tau \in K \mid \gamma=y) \, \sP_\gamma(\rd y)},
        \end{align*}
        which completes the proof.
 \end{enumerate}
  \end{proof}

\section{Outlook: Statistical applications}\label{sec:application}

In univariate extreme value theory, a standard method for estimating
the extreme value parameters fits all data exceeding a high threshold
to a certain Poisson point process. This peaks-over-threshold approach
has been generalized in \cite{roo2006} to the multivariate setting.
Therein, generalized multivariate Pareto distributions are obtained as
the max-limit of some multivariate random vector in the MDA of an
extreme value distribution by conditioning on the event that at least
one of the components is large.  Conditioning on the same extremal
events, the recent contribution \cite{fal2012} analyzes the asymptotic
distribution of exceedance counts of stationary sequences.  \\ Here,
we have suggested conditioning a stochastic process $\eta(t) : t\in
T\}$ in the MDA of a max-stable process $\{\xi(t) : t\in T\}$ such
that it converges to the incremental processes $W$ in \eqref{def_xi}
or the shape functions $F$ in \eqref{def_M3}. In this section we
provide several examples how these theoretical results can be used for
statistical inference. The approach is based on a multivariate
peaks-over-threshold method for max-stable processes, though the
definition of extreme events differs from that in
\cite{roo2006,fal2012}.\\ In the sequel, suppose that
$\eta_1,\dots,\eta_n,$ $n\in\Nb,$ are independent observations of the
random process $\eta$, already normalized to standard Pareto margins.

\subsection{Incremental representation}

For a max-stable process $\xi$ that admits an incremental representation
\begin{align}
  \label{def_xi_again}
  \xi(t) = \max_{i\in\Nb} U_i W_i(t), \quad t\in T,
\end{align}
as in \eqref{def_xi}, the statistical merit of the convergence results
in Theorem \ref{theo_cond_increments_general} and Theorem 
\ref{theo_cond_increments_cont} is the 
``deconvolution'' of $U$ and $W$ which allows to substitute
estimation of $\xi$ by estimation of the process $W$.
As only the single \emph{extreme} events converge to $W$, we define 
the index set of extremal observations as
\begin{align*}
  I_1(n) = \bigl\{ i\in \{1,\dots n\}: \ \eta_i(t_0) > a(n) \bigr\},
\end{align*}
for some fixed $t_0\in T$. The set $\{ \eta_i(\cdot) / \eta_i(t_0) :
\ i \in I_1(n) \}$ then represents a collection of independent random
variables that approximately follow the distribution of $W$. Thus,
once the representation in \eqref{def_xi_again} is known, both
parametric and non-parametric estimation for the process $W$ is
feasible. For statistical inference it is necessary that the number of
extremal observations $|I_1(n)|$ converges to $\infty$, as
$n\to\infty$. This is achieved by choosing the sequence of thresholds
$a(n)$ according to Remark \ref{rem_thres}.

\begin{ex}[Symmetric logistic distribution, cf.\ Example \ref{ex:symm_log}]
The dependence parameter $q\geq 1$ of the symmetric logistic
distribution \eqref{eq:cdf_symm_log} can be estimated by perceiving
the conditional increments of $\eta$ in the MDA as realizations of $W$
and maximizing the likelihood
\begin{align*}
  &\sP\big(W(t_1) \in \rd s_1, \dots W(t_k) \in \rd s_k \,\big|\, q\big) \\
  &=    \left(1+\textstyle\sum_{i=1}^k s_i^{-q}\right)^{1/q -(k+1)}
    \left(\textstyle\prod_{i=1}^k(iq-1)\right)
    \textstyle\prod_{i=0}^k s_i^{-q-1} \ \rd s_1\dots \rd s_k.
\end{align*}
\end{ex}

\begin{ex}[Brown-Resnick processes]
Recall that the Brown-Resnick processes in Example \ref{BRproc} admit
a representation \eqref{def_xi} with log-Gaussian incremental process
$W(t) = \exp\left\{Y(t) - \sigma^2(t) / 2\right\}, t\in\R^d$.  Hence,
standard estimation procedures for Gaussian vectors or processes can
be applied for statistical inference. \cite{eng2012a} explicitly
construct several new estimators of the variogram $\gamma$ based on
the incremental representation, which also covers H\"usler-Reiss
distributions, and they provide some basic performance analyses.
\end{ex}

\subsection{Mixed moving maxima representation}

Similarly, in case of the mixed moving maxima representation
\begin{align}
 M(t) = \max_{i \in \Nb} U_i F_i(t - S_i), \quad t \in \R^d,
\end{align}
the convergence results of Theorem \ref{thm_conv_mmm} can be used to
estimate $F$ (or $F_1 = F / \lambda$) on some compact domain $K$
instead of estimating $M$ directly.  Here, the index set $T$ of the
observed processes $\{\eta_i(t): \ t \in T\}$, $i=1,\ldots, n$, can be
identified with $\Q \oplus L$ from Theorem \ref{thm_conv_mmm}.  The
set $L$ should be sufficiently large such that it is reasonable to
assume that the components $\{U_i F_i(\cdot - S_i): \ S_i \notin \Q
\oplus L\}$ hardly affect the process $M$ on $\Q \oplus K$ (that is,
$\mu(C_B \cap A_L) / \mu(A_L) \approx \sP(F(\cdot) - \lambda \in B)$
in the proof of Theorem \ref{thm_conv_mmm}).  At the same time, a
large set $\Q$ leads to a rich set of usable observations $ \widetilde
F_1^{(i)} = \eta_i(\tau_\Q(\eta_i)+\cdot) / \eta_i(\tau_\Q(\eta_i)),
\ i \in I_2(n),$ where
$$ I_2(n) = \left\{ i \in \{1,\ldots,n\}: \ \max_{t\in \Q}\eta_i(t) =
\max_{t \in \Q \oplus L} \eta_i(t) \geq a(n)\right\}.$$ The resulting
processes $\widetilde F_1^{(i)},\ i \in I_2(n),$ can be interpreted as
independent samples from an approximation to $F_1$.  This approach can
be expected to be particularly promising in case of F having a simple
distribution or even being deterministic.

\begin{ex}[M3 processes with deterministic shape functions]
Some examples of mixed moving maxima processes have already been
analyzed for statistical inference by \cite{deh2006} who use normal,
exponential and t densities as shape functions. More precisely, they
consider M3 models with
\begin{align}
 F_1(t) &= \exp\left\{- \frac{\beta^2 t^2} 2\right\}, \quad  &\lambda {}={}& \frac{\beta}{\sqrt{2\pi}}, \label{eq:mmm-ex-1}\\
 F_1(t) &= \exp\left\{- \beta |t|\right\}, \quad  &\lambda {}={}& \frac{\beta}{2}, \label{eq:mmm-ex-2}\\
 \text{and} \quad F_1(t) &= \left( 1 + \frac{\beta^2 t^2}{\nu} \right)^{- \frac{\nu+1} 2}, \quad 
 &\lambda {}={}& \frac{\beta \Gamma\left( \frac{\nu+1} 2\right)}{\sqrt{\pi\nu} \Gamma\left(\frac \nu 2\right)}\, \quad \nu > 0, \label{eq:mmm-ex-3}
\end{align}
all parametrized by $\beta > 0$. \cite{deh2006} introduce
consistent and asymptotically normal estimators based on the interpretation of
$\beta$ as a dependence parameter. From the samples $\widetilde F_1^{(i)}$, $i \in I_2(n)$,
we get a new estimator
\begin{equation*}
  \widehat F_1 = \frac 1 {|I_2(n)|} \sum_{i \in I_2} \widetilde F_1^{(i)}
\end{equation*}
for $F_1$. Applying this estimator, $\beta$ can be estimated by a least squares fit
of \eqref{eq:mmm-ex-1}--\eqref{eq:mmm-ex-3} to $\widehat F_1$ at some locations $t_1, \ldots t_m \in K$.
Note that in case of the normal model \eqref{eq:mmm-ex-1} and the exponential model \eqref{eq:mmm-ex-2}, the
logarithm of the shape function $F_1$ depends linearly on $\beta^2$ and $\beta$,
respectively, and $\log \widehat F_1$ can be fitted by ordinary least
squares.
\end{ex}

\begin{ex}[Brown-Resnick processes]
The mixed moving maxima representation can also be employed for
estimation of Brown-Resnick processes although the distribution of $F$
is much more sophisticated than the one of $W$ in the incremental
representation (cf. \cite{eng2011, oes2012}). A relation between the
shape function $F$ and the variogram $\gamma$ of the Brown-Resnick
process can be obtained via the \emph{extremal coefficient function}
$\theta(\cdot)$. For a stationary, max-stable process $\xi$ with identically distributed marginals, \cite{sch2003} defined 
the extremal coefficient function $\theta$
via the relation
\begin{align*}
  \sP(\xi(0)\leq u,\, \xi(h) \leq u) = \sP(\xi(0) \leq u)^{\theta(h)}, 
  \quad h \in \R^d.
\end{align*}
For mixed moving maxima processes, we have
\begin{align} \label{eq:theta-mmm}
  \theta(h)  
  =\sE\left. \int_{\R^d} \{ F(t) \vee F(t+h)\} \sd t\right.
  =\frac{\sE\left. \int_{\R^d} \{ F_1(t) \vee F_1(t+h)\} \sd t\right.}{
    \sE\left.\int_{\R^d} F_1(t) \sd t\right.}
\end{align}
and, at the same time, for Brown-Resnick processes \cite{kab2009},
\begin{align} \label{eq:theta-br}
 \theta(h) = 2\Phi\left(\sqrt{\gamma(h)}/ 2\right),
\end{align}
where $\Phi$ is the standard Gaussian distribution function.
Identifying \eqref{eq:theta-mmm} with \eqref{eq:theta-br} and plugging
in the samples $\widetilde F_1^{(i)}$, $i \in I_2(n)$, we get the variogram estimator
\begin{align*}
 \widehat \gamma(h) = 
 \left\{2 \Phi^{-1}\left( 
 \frac{\displaystyle \sum_{i \in I_2(n)} \int_{\widetilde K} \widetilde F_1^{(i)}(t) 
     \vee \widetilde F_1^{(i)}(t+h) \sd t}{2 \displaystyle \sum_{i \in I_2(n)} \int_{\widetilde K} \widetilde F_1^{(i)}(t) \sd t} 
 \right)\right\}^2,
\end{align*}
where $\widetilde K$ is a large set such that $\widetilde K, \widetilde K +h \subset K$.
\end{ex}

\appendix

\section{Auxiliary Results for the Proof of Theorem \ref{thm_conv_mmm}}

\begin{lem} \label{lem:AL}
$A_L$ is closed. The set $A_L^*$ is in the interior of $A_L$.
\end{lem}
\begin{proof}
 The first assertion is obvious. For the second one, let $f^* \in
 A_L^*$. Then, we have $f^*(\tau_\Q(f^*|_\Q)) =: \alpha >
 1$. Furthermore, there is $\delta > 0$ such that
 $B_\delta(\tau_\Q(f^*|_\Q)) = \{t \in \R^d: \ ||t - \tau_\Q(f|_\Q)||
 < \delta \} \in \Q^o$ and we have
 \begin{equation}
  \beta := \sup_{t \in \Q \oplus L \setminus B_\delta(\tau_\Q(f^*|_\Q))} f^*(t) - \max_{t\in \Q}f^*(t) < 0.
 \end{equation}
 Now, we choose $\varepsilon < \min \{ \frac {\alpha-1} 2, \frac {|\beta|} 2\}$ and show that $B_\varepsilon(f^*) = \{ f \in C(\Q \oplus L):\ ||f - f^*||_\infty < \varepsilon\} \subset A_L$.
 This holds, as for any $f \in B_\varepsilon(f^*)$, we have
 \begin{align*}
  & f(\tau_\Q(f|_\Q)) \geq f(\tau_\Q(f^*|_\Q)) > \alpha - \varepsilon > \frac {1+\alpha }2 > 1\\
  \text{and} \quad & \max_{t\in \Q}f(t) \leq \max_{t \in \Q\oplus L} f(t) =  \max\left\{\max_{t \in \Q \oplus L \setminus \Q^o} f(t), \ \max_{t\in \Q}f(t)\right\}\\
                   & {}\leq{} \max\left\{ \beta + \alpha + \varepsilon, \ \max_{t\in \Q}f(t)\right\} {}\leq{} \max\left\{ \alpha - \varepsilon, \ \max_{t\in \Q}f(t)\right\} {}={} \max_{t\in \Q}f(t),
 \end{align*}
 which means equality.
  \end{proof}

\begin{lem} \label{lem:interCB}
 The set $C_B^*$  is in the interior of $C_B$.
\end{lem}
\begin{proof}
 Let $f^* \in C_B^*$.  Then, $t^* = \argmax_{t \in \Q} f^*(t)$ is well-defined and necessarily, as $f \geq 0$, 
 \begin{equation} \label{eq:alpha}
   \alpha := f^*(t^*)  \in (0,||f^*||_\infty].  
 \end{equation}
 Since $f^*(t^* + \cdot) / f^*(t^*) \in B^o$, there is some $\varepsilon > 0$ such that
 \begin{equation} \label{eq:ball}
 \left\{f \in C(K): \ \left|\left|f^*(t^* + \cdot) \Big/ f^*(t^*) - f\right|\right|_\infty < \varepsilon\right\} \subset B. 
 \end{equation}
 Furthermore, $f^*$ is uniformly continuous on the compact set $\Q \oplus L$, i.e.\ there exists some $\delta > 0$ such that
 \begin{equation} \label{eq:unifcont}
  \sup_{s,t \in \Q \oplus L, \ ||s-t|| < \delta} |f^*(s) - f^*(t)| < \frac \varepsilon 3 \alpha.
 \end{equation}
 Then, as $\argmax_{t \in \Q} f^*(t)$ is unique, we have that
 \begin{equation} \label{eq:beta}
  \beta := \max_{t \in \Q \setminus \{t \in \R^d: \ ||t-t^*|| < \delta\}} f^*(t)  - f^*(t^*) \in [-\alpha,0).
 \end{equation}
 Choose $\varepsilon^* < \min\left\{\frac {|\beta|} {2\alpha}, \frac \varepsilon 6 \frac \alpha {||f^*||_\infty}\right\}$.
 We will show that $B_{\varepsilon^*\alpha}(f^*) =\{f \in C(\Q \oplus L): \ ||f-f^*||_\infty < \varepsilon^*\alpha\} \subset C_B$.
 To this end, let $f_0 \in B_{\varepsilon^*\alpha}(f^*)$. Then, because of Equation \eqref{eq:beta} and $\varepsilon^*\alpha < \frac {|\beta|} 2$, we have that
 $||t_0 - t^*|| \leq \delta$ for $t_0 = \tau_\Q(f_0|_\Q)$.
 Therefore,
 \begin{align}
  & \sup_{t \in K} \left| \frac{f^*(t^*+ \cdot)}{f^*(t^*)} - \frac{f_0(t_0 + \cdot)}{f_0(t_0)} \right| \notag\\
 \leq{} & \sup_{t \in K} \left| \frac{f^*(t^*+\cdot)}{f^*(t^*)} - \frac{f^*(t_0+\cdot)}{f^*(t^*)}\right| 
        + \sup_{t \in K} \left| \frac{f^*(t_0+\cdot)}{f^*(t^*)} - \frac{f^*(t_0+\cdot)}{f_0(t_0)}\right| \notag\\
      &  + \sup_{t \in K} \left| \frac{f^*(t_0+\cdot)}{f_0(t_0)} - \frac{f_0(t_0+\cdot)}{f_0(t_0)}\right| \qquad
 \leq \qquad  \frac \varepsilon 3 + \frac{\varepsilon^*}{1-\varepsilon^*} \frac{||f^*||_\infty}{\alpha} + \frac{\varepsilon^*}{1-\varepsilon^*}, \label{eq:dist}
 \end{align}
 where we used Equation \eqref{eq:unifcont} and the fact that $f_0 \in B_{\varepsilon^*\alpha}(f^*)$.
 Equations \eqref{eq:alpha} and \eqref{eq:beta} and the choice of $\varepsilon^*$ yield that
 $\varepsilon^*/(1-\varepsilon^*) \leq (\varepsilon^* ||f^*||_\infty)/((1-\varepsilon^*)\alpha) \leq  2\varepsilon^*||f^*||_\infty / \alpha  < \varepsilon/3$,
 i.e.\ each summand on the right-hand side of \eqref{eq:dist} is smaller than $\varepsilon/3$. Thus,
$f_0(t_0 + \cdot) / f_0(t_0) \in \{f \in C(K): \ ||f^*(t^* + \cdot) / f^*(t^*) - f||_\infty < \varepsilon\} \subset B$ by Equation
\eqref{eq:ball} and $f_0 \in C_B$.
  \end{proof}

\begin{lem} \label{lem:CB}
 The closure of $C_B$ is a subset of 
 \begin{align*}
 B^* \cup{} & \left\{f \in C(\Q \oplus L): \, \argmax_{t \in \Q} f(t) \text{ is not unique} \right\}\\
     \cup{} & \left\{f \in C(\Q \oplus L): \, f(\tau_\Q(f|_\Q) + \cdot) \Big/ f(\tau_\Q(f|_\Q)) \in \partial B \right\}.
 \end{align*}
\end{lem}
\begin{proof}
 Let $\{f_n\} \subset C_B$ be a sequence converging uniformly to some $f^* \in C(\Q\oplus L)$. We have to verify that $f^*(\tau_\Q(f^*|_\Q) + \cdot) / f^*(\tau_\Q(f^*|_\Q)) \in B \cup \partial B$
 if $\argmax_{t \in \Q} f^*(t)$ is unique.
 Analogously to the proof of Lemma \ref{lem:interCB} we can show that for any $\varepsilon_2 > 0$ there is some $\varepsilon_1 > 0$ such that
 \begin{align*}
  & ||f-f^*||_{\infty,\Q \oplus L} < \varepsilon_1\\
  {}\Longrightarrow{} \quad &  \left|\left|\frac{f(\tau_\Q(f|_\Q)+\cdot)}{f(\tau_\Q(f|_\Q))} - \frac{f^*(\tau_\Q(f^*|_\Q)+\cdot)}{f^*(\tau_\Q(f^*|_\Q))}\right|\right|_{\infty,K} < \varepsilon_2.
 \end{align*}
 Thus, $f_n(\tau_\Q(f_n|_\Q) + \cdot)/f_n(\tau_\Q(f_n|_\Q))$ converges to $f^*(\tau_\Q(f^*|_\Q) + \cdot)/f^*(\tau_\Q(f^*|_\Q))$ in $C(K)$. Hence, as $B \cup \partial B$ is closed,
 $f^*(\tau_\Q(f^*|_\Q) + \cdot)/f^*(\tau_\Q(f^*|_\Q)) \in B \cup \partial B$.
  \end{proof}

\section{Proof of Example \ref{ex:mmm-mda}}\label{sec:proof_ex_mmm-mda}

\begin{proof}[Proof of Example \ref{ex:mmm-mda}]
Let $\tilde M_j$, $j\in\Nb$, be independent copies of the process
$\tilde M$ and consider 
$$M_n(\cdot) = \frac{1}{cn}  \max_{i=1}^n \tilde M_i(\cdot).$$
Further, suppose that $L \subset \R^d$ is an arbitrary compact set.
Note that by Remark \ref{weak_conv_Rd} it suffices to show weak convergence
of $M_n \cvgdist M$, $n\to\infty$, on $C(L)$.

\medskip

To prove the first assertion, note that,
for $t_0, \ldots, t_k\in\R^d$, $s_0, \ldots, s_k\geq 0$, $k\in\Nb$,
we have
\begin{align}
&\sP(M_n(t_0)\leq s_0, \ldots, M_n(t_k)\leq s_k) \notag\\
&= \left[\1_{\kappa \leq \min_{l=0}^k cns_l }\cdot \sP\left\{ \Pi_{c, \epsilon} 
  \left(\left\{(u,t,f): 
  \max_{l=0}^k u f(t_l-t)/(cns_l) > 1\right\}\right) = 0\right\}\right]^n \notag\\
&= \1_{\kappa \leq \min_{l=0}^k cns_l }\cdot \exp\left(- n\int_{C(\R^d)} \int_{\R^d} 
\min\left\{\frac{1}{\epsilon}, \max_{l=0}^k \frac{f(t_l-t)}{cns_l}\right\}\, c\sd t \, \sP_F(\rd f) \right)\notag\\
&\longrightarrow 
\exp\left(- \int_{C(\R^d)} \int_{\R^d} 
\max_{l=0}^k (f(t_l-t)/s_l) \sd t \, \sP_F(\rd f) \right), \label{proofMDAM3}
\end{align}
as $n\to\infty$, where the convergence holds due to monotone
convergence.  The right-hand side of \eqref{proofMDAM3} coincides
with the marginal distribution of $M$ (cf.\ \eqref{M3_marginal}).

\medskip

For convergence of $M_n$ to $M$ in the sense of weak convergence in $C^+(L)$ endowed with the
  topology of uniform convergence, it remains to show that the sequence of restricted processes $\{M_n|_L : n\in \Nb\}$ is tight.
  To this end, by \cite[Thm. 7.3]{bil1999}, it suffices to verify
  that for any $\varepsilon > 0$, $\eta \in (0,1)$, there exist $\delta > 0$, $n_0 \in \Nb$ such that
  $$ \sP\left\{ \sup_{||s-t|| < \delta} |M_n(s) - M_n(t)| \geq \varepsilon\right\} \leq \eta, \qquad n \geq n_0.$$
  By Equation \eqref{eq:sup-integrability-cont}, we can choose
  $R > 0$ such that 
  \begin{equation} \label{eq:smallintegral}
   \int_{\R^d \setminus (L \oplus B_R(\0))} \sE\left( \sup_{t \in L} F(t-s) \right) \sd s < \frac {\varepsilon \eta} 2,
  \end{equation}
  where $B_R(\0) = \{ x \in \R^d: \ ||x|| \leq R\}$.
  Furthermore, \eqref{eq:sup-integrability-cont} implies that $\sE\left( \sup_{t \in K} F(s)\right) < \infty$ for any compact set $K \subset \R^d$.
  Therefore, as each realization of $F$ is uniformly continuous on $B_{R+d(L)}(\0)$, where $d(L) = \sup_{s_1,s_2 \in L} ||s_1-s_2||$ denotes the diameter of $L$,
  dominated convergence yields
  \begin{equation*}
   \lim_{\delta \searrow 0} \sE\left( \sup_{s,t \in B_{R+d(L)}(\0), \ ||s-t|| < \delta} |F(s) - F(t)|\right) = 0.
  \end{equation*}
  In particular, we can choose $\delta > 0$ such that
  \begin{equation} \label{eq:delta}
   \sE\left( \sup_{s_1,s_2 \in B_{R+d(L)}(\0), \ ||s_1-s_2|| < \delta} |F(s_1) - F(s_2)|\right) < \frac {\varepsilon \eta}{2 |L \oplus B_R(\0)|}.
  \end{equation}
  Then, we get
  \begin{align*}
   & \sP\left\{ \sup_{||s_1-s_2|| < \delta, \ s_1,s_2 \in L} | M_n(s_1) - M_n(s_2) | \geq \varepsilon \right\} \\
   \leq{} & n \sP\left\{ \sup_{||s_1-s_2|| < \delta, \ s_1,s_2 \in L} |\tilde M_n(s_1) - \tilde M_n(s_2) | \geq c n\varepsilon \right\} \\
   \leq{}& n \bigg( \sP\bigg\{ \Pi\bigg(\bigg\{ (u,t,f): \ t \in L \oplus B_R(\0), \\
             & \hspace{3cm} \sup_{\substack{s_1,s_2 \in B_{R+d(L)}(\0),\\ ||s_1-s_2|| < \delta} } |f(s_1)-f(s_2)| > \frac{cn\varepsilon}{u}
             \bigg\}\bigg) > 0\bigg\}\\
         & + \sP\bigg\{ \Pi\bigg(\bigg\{ (u,t,f): \ t \in \R^d \setminus (L \oplus B_R(\0)), 
\sup_{s \in L} |f(s-t)|  > \frac{cn\varepsilon}{u}  \bigg\} \bigg) > 0\bigg\}\bigg) \displaybreak[0]\\
   \leq{} & n\bigg( 1 - \exp\bigg\{ - \int_{L \oplus B_R(\0)} \int_\epsilon^\infty u^{-2} \\
       & \hspace{3.2cm}  \cdot \sP\bigg( \sup_{\substack{s_1,s_2 \in B_{R+d(L)}(\0),\\ ||s_1-s_2|| < \delta} } |F(s_1)-F(s_2)| > \frac{c n\varepsilon}{u}\bigg) \sd u \, c \sd t \bigg\}\\
       & +  1- \exp\bigg\{ - \int_{\R^d \setminus (L \oplus B_R(\0))} \int_\epsilon^\infty u^{-2} 
        \sP\bigg( \sup_{s \in L} |F(s-t)| > \frac{c n\varepsilon}{u}\bigg) \sd u \, c \sd t \bigg\} \bigg) \displaybreak[0]\\
   \leq{} & n\bigg( 1 - \exp\bigg( - \frac{|L \oplus B_R(\0)|}{n \varepsilon} 
                      \sE\bigg\{ \sup_{\substack{s_1,s_2 \in B_{R+d(L)}(\0),\\ ||s_1-s_2|| < \delta} } |F(s_1)-F(s_2)| \bigg\} \bigg)\\
            & + 1 - \exp\bigg(- \frac 1 {n\varepsilon} \int_{\R \setminus (L \oplus B_R(\0))} \sE\left\{ \sup_{s \in L} |F(s-t)| \right\} \sd t \bigg) \bigg)\\
   \leq{} & n\left( 1 - \exp\left(- \frac \eta {2n}\right) + 1 -  \exp\left(- \frac \eta {2n}\right)\right) \quad \leq \quad \eta,
  \end{align*}
  where we used Equation \eqref{eq:delta} and \eqref{eq:smallintegral}.
  Thus, the sequence of processes $\{M_n|_L : n\in \Nb\}$ is tight.
\end{proof}

\acks The authors are grateful to Zakhar Kabluchko for useful
suggestions and hints.  S. Engelke has been financially supported by
Deutsche Telekom Stiftung. A. Malinowski has been financially
supported the German Science Foundation (DFG), Research Training Group
1644 `Scaling problems in Statistics'.  M. Oesting and M. Schlather
have been financially supported by Volkswagen Stiftung within the
project `WEX-MOP'.

\bibliographystyle{apt}  
\small
\bibliography{HREstimation}{}

\end{document}